\documentclass[11pt,a4paper,reqno]{amsart}

\usepackage{amsmath}
\usepackage{amsthm}
\usepackage{amsfonts}
\usepackage{amssymb}

\newcommand\R{{\mathbf{R}}}
\newcommand\C{{\mathbf{C}}}
\newcommand\Z{{\mathbf{Z}}}

\newcommand\ZN{\mathbf{Z}_{N}}
\newcommand\GN{\mathbf{G}_{N}}
\newcommand\Gk{\mathbf{G}_{k}}
\newcommand\Dphi{\mathcal{D}\phi}
\newcommand\Df{\mathcal{D}f}
\newcommand\A{\mathcal{A}}
\renewcommand\P{{\mathcal{P}}}
\renewcommand\d{\textrm{deg}}

\newcommand\E{{\mathbf{E}}}
\newcommand\F{{\mathcal{F}}}
\newcommand\Fq{{\mathbf{F}_q}}

\newcommand\FqN{{\mathbf{F}_{q^{N}}}}

\theoremstyle{plain}
  \newtheorem{theorem}{Theorem}
  \newtheorem{conjecture}{Conjecture}

  \newtheorem{proposition}{Proposition}
  
  \newtheorem{lemma}{Lemma}
  \newtheorem{claim}{Claim}
  \newtheorem{corollary}{Corollary}[theorem]

\theoremstyle{remark}
  \newtheorem{remark}[subsection]{Remark}
  \newtheorem{remarks}[subsection]{Remarks}

\theoremstyle{definition}
  \newtheorem{definition}{Definition}

\include{psfig}

\begin{document}

\title{Green-Tao Theorem in function fields}

\author{Th\'ai Ho\`ang L\^e}
\address{UCLA Department of Mathematics, Los Angeles, CA 90095-1596.}
\email{leth@math.ucla.edu}

\begin{abstract}
We adapt the proof of the Green-Tao theorem on arithmetic progressions in primes to the setting of polynomials over a finite fields, to show that for every $k$, the irreducible polynomials in $\Fq[t]$ contains configurations of the form $\{f+ Pg : \d(P)<k \}, g \neq 0$.
\end{abstract}

\maketitle


\section{Introduction}
In \cite{gt-primes}, Green and Tao proved the following celebrated theorem now bearing their name:

\begin{theorem}[Green-Tao]
The primes contain arithmetic progressions of arbitrarily length. Furthermore, the same conclusion is true for any subset of positive relative upper density of the primes.
\end{theorem}
Subsequently, other variants of this theorem have been proved. Tao and Ziegler \cite{tao-ziegler} proved the generalization for polynomial progressions $a+p_1(d),\ldots, a+p_{k}(d)$, where $p_{i} \in \Z[x]$ and $p_{i}(0)=0$. Tao \cite{tao-g} proved the analog in the Gaussian integers.

It is well known that the integers and the polynomials over a finite field share a lot of similarities in many aspects relevant to arithmetic combinatorics. Therefore, it is natural, as Green and Tao did, to suggest that the analog of this theorem should hold in the setting of function fields:

\begin{conjecture}
For any finite field $\mathbf{F}$, the monic irreducible polynomials in $\mathbf{F}[t]$ contain affine spaces of arbitrarily high dimension.
\end{conjecture}

We give an affirmative answer to this conjecture. More precisely, we will prove:

\begin{theorem}[Green-Tao for function fields] \label{FFGT}
Let $\Fq$ be a finite field over $q$ elements. Then for any $k>0$, we can find polynomials $f, g \in \Fq[t], g \neq 0$ such that the polynomials $f+Pg$, where $P$ runs over all polynomials $P \in \Fq[t]$ of degree less than $k$, are all irreducible. Furthermore, such configurations can be found in any set of positive relative upper density among the irreducible polynomials.
\end{theorem}

Here we define the upper density of a set $\A \subset \Fq[t]$ to be $\overline{d}(\A)=\overline{\lim}_{N \rightarrow \infty} \frac{ \# \{ f \in \A, \d(f) < N \} }{q^{N}}$, and the relative upper density of $\A$ in the set $\P$ of all irreducible polynomials to be $\overline{d}_{\P}(\A)=\overline{\lim}_{N \rightarrow \infty} \frac{ \# \{ f \in \A, \d(f) < N \} }{\# \{ f \in \P, \d(f) < N \}}$. The conjecture then follows since the monic polynomials is of positive density in all the polynomials.

Our arguments follow Green-Tao's very closely. We also have chosen to incorporate some modifications that simplify considerably some major steps in the original arguments. Therefore, the paper may prove to be helpful to those who want to understand the ideas of the proof of Green-Tao's theorem.

\textbf{Acknowledgements.} I am grateful to my advisor Terence Tao for suggesting me this project, frequent consultation and assistance throughout the preparation of this paper.

\section{Outline of the proof and notation}

\subsection{Notation}
Through out the paper, we will be working with a fixed field $\Fq$ on $q$ elements, where $q$ is a prime power. Let $\Fq[t]$ be the ring of polynomials with coefficients in $\Fq$. Let $\Fq(t)$ be the quotient ring of $\Fq[t]$, i.e. $\Fq(t)=\{\frac{f}{g}| f,g \in \Fq[t], g\neq 0 \}$. The value of $k$ will be kept fixed. We will keep our notation consistent with that of Green and Tao.

It is clear that it suffices to find affine spaces in the irreducible polynomials of degree smaller than $N$ where $N$ is sufficiently large. Denote by $\GN$ the set of all polynomials in $\Fq[t]$ of degree less than $N$. A priori, $\GN$ is an additive group. We also, for each $N$, fix a monic irreducible polynomial $f_{N}\in\Fq[t]$. Then the additive group $\GN$ can be endowed by a field structure isomorphic to $\FqN$, the field on $q^{N}$ elements via multiplication modulo $f_{N}$. The need for the field structure arises in the same way as when we convert $\{1, \ldots, N\}$ into $\ZN$, the main reason being that we can freely perform divisions. Of course, there is a price to pay, namely the ``wraparound'' effect, which arises when we want to pull things back from $\FqN$ to $\GN$, but this is easy to deal with. In the setting of $\Fq[t]$ this is even simpler, since addition of polynomials does not increase the maximum of the degrees, quite contrarily to the integers.

If $\phi$ is a function on a finite set $A$, we write $\E( \phi(x) | x \in A)$, or $\E_{x \in A}\phi(x)$, or simply $\E_{A}\phi$ to denote the expectation of $\phi$ on $A$, in other words the average value of $\phi$ on $A$. We denote the inner product of two functions $\phi, \psi$ on $A$ as $\langle \phi, \psi \rangle= \E_{x \in A} \phi(x) \psi(x)$. We will also define the $L^{p}$-norm of $\phi$ to be $\|\phi \|_{p}=\E \left( |\phi(x)|^{p} | x \in A\right)^{1/p}$, and the $L^{\infty}$-norm to be $\|\phi \|_{\infty}=\sup_{x \in A} |\phi(x)|$.

Let $K= |\Gk|=q^{k}$, the number of polynomials of degree less than $k$.

For a non-zero polynomial $f \in \Fq[t]$ define the norm of $f$ to be $|f|=q^{\d(f)}$. Also, let $|0|=0$. Then the norm $|\cdot|$ defines a distance on $\Fq[t]$. Often, when dealing with the wraparound effect, we will make use of cylinder sets. A cylinder set of radius $r$ is simply the set of all $f \in \Fq[t]$ whose distance to a given point is at most $r$. The cylinder sets are the analog of intervals in $\R$, but they enjoy a more pleasant property that for any two cylinders, either they are disjoint or one is contained in the other.

Let us call a set $\{f+Pg: P \in \Gk \}$ a $k$-configuration. If $g \neq 0$ then it is called a non-trivial $k$-configuration. A $k$-configuration in $\GN$ is necessarily a $k$-configuration in $\FqN$, but not vice versa.

For two quantities $A,B$, we write $A=O(B)$, or $A\ll B$, or $B \gg A$ if there is an absolute constant $C$ such that $|A| \leq CB$. If $A$ and $B$ are functions of the same variable $x$, we write $A=o_{x \rightarrow \infty}(B)$ if $A/B$ tends to 0 as $x$ tends to infinity. If the constant $C$, (respectively, the rate of convergence of $A/B$) depends on a parameter, e.g. $m$, then we write $A=O_{m}(B)$ (respectively, $A=o_{m; x\rightarrow \infty}$). Dependence on fixed quantities such as $q$ or $k$ will be often omitted. Most of the time we will be dealing with functions in $N$, and when it is clear we will remove it from the notation. Thus $O(1)$ stands for a bounded quantity (independent of $N$) and $o(1)$ stands for a function that goes to 0 as $N$ tends to infinity.

\subsection{Outline of the proof}
The starting point of Green-Tao is Szemer\'edi's theorem, which
states that any set of positive density among the natural numbers
contains arbitrarily long arithmetic progressions. Actually, they
needed a stronger form of Szemer\'edi's theorem, obtained by
incorporating an argument known as Varnavides's trick
\cite{varnavides}. In the setting of function fields, an analog of
Szemer\'edi's theorem is readily available \cite{fk}, \cite{blm}.
Coupled with the Varnavides argument, this gives the following
result, which we will prove in Section \ref{S3}:

\begin{theorem}[Szemer\'edi for function fields] \label{FFSzemeredi}
For every $\delta>0$, there exists a constant $c(\delta)>0$ such that, for every function $ \phi : \FqN \rightarrow \R$ such that
$0 \leq \phi(x) \leq 1$ for all $x$ and $\E(\phi|\FqN) \geq \delta$ , we have
$$\E \left(\prod_{P \in \Gk} \phi(f+Pg)| f, g \in \FqN \right) \geq c(\delta)$$
\end{theorem}

Following Green and Tao, our next step is a transference principle, which allows us to generalize Szemer\'edi's theorem to larger classes of $\phi$, whose functions are not necessarily bounded. Let us call a measure a function $\nu :\GN \rightarrow \R$. A pseudorandom measure is a measure satisfying two technical conditions (to be defined later in Section \ref{S4}), called the linear forms condition and the correlation condition.

\begin{theorem} [Green-Tao-Szemer\'edi for function fields]\label{FFESzemeredi}
Given a pseudorandom measure $\nu :\FqN \rightarrow \R$. Then for every $\delta>0$, there exists a constant $c'(\delta)>0$
such that, for every function $ \phi : \FqN \rightarrow \R$ such that $0 \leq \phi(f) \leq \nu(f)$ for all $f$ and $\E(f|\FqN)\geq \delta$ , we have
$$\E \left( \prod_{P \in \Gk} \phi(f+Pg)| f, g \in \FqN \right) \geq c'(\delta)-o(1)$$
\end{theorem}

This is obtained by means of a decomposition result, namely any function $\phi$ bounded by a pseudorandom measure can be decomposed as $\phi=\phi_{1}+\phi_{2}$, where $\phi_1$ is a nonnegative, bounded function, whose average is bounded from below, and $\phi_2$ is uniform in the sense that it is small in a norm (the Gowers norm to be defined later) that is relevant to counting $k$-configurations. Thus the contribution of $f_2$ in $\E \left( \prod_{P \in \Gk} \phi(f+Pg)| f, g \in \FqN \right)$ is small, so that the latter is close to $\E \left( \prod_{P \in \Gk} \phi_1(f+Pg)| f, g \in \FqN \right)$, which is bounded from below by the usual Szemer\'edi's theorem. The proof of the decomposition result in Green-Tao \cite{gt-primes} and later in \cite{tao-ziegler} is quite involved. Recently Gowers \cite{gowers} and Reingold-Trevisan-Tulsiani-Vadhan \cite{rttv}, \cite{rttv2} have found much simpler proofs of this result, the main tool being the Hahn-Banach theorem. Moreover, their formulations of the result are very general and directly applicable to our setting of function fields.

Once Theorem \ref{FFESzemeredi} is established, the final step is to show that $\nu$ can be constructed in such a way that $\nu$ majorizes functions supported on irreducible polynomials, such as (variants of) the von Mangoldt function $\Lambda$\footnote{Ideally, we would like to take $\nu=\Lambda$, but to verify that $\Lambda$ satisfies the conditions of a pseudorandom measure we need more information about additive properties of the irreducible polynomials, which would in turn lead to statements analogous to the twin prime conjecture. Even if the analog of the twin prime conjecture in $\Fq[t]$ is known to be true, the conjectured asymptotic formula is not yet proven.}, where $\Lambda(f)=
\left\{
                                         \begin{array}{ll}
                                           \d(P), & \hbox{if $f=cP^{k}$, where P is irreducible and $c\in \Fq$;} \\
                                           0, & \hbox{otherwise.}
                                         \end{array}
                                       \right.$
To this end, we will make adaptations of the truncated divisor sum
of Goldston and Y{\i}ld{\i}r{\i}m on their work on short gaps
between primes \cite{gy}, \cite{gy2} in Section \ref{S9}.

\begin{theorem}[Goldston-Y{\i}ld{\i}r{\i}m for function fields]\label{FFGY}
For any $\A \subset \P$ such that $\overline{d}_{\P}(\A)>0$, there
exist a constant $\delta >0$, a pseudorandom measure $\nu :\FqN
\rightarrow \R$, a function $\phi :\FqN \rightarrow \R$ and $W,b \in
\FqN$ such that the following are true for infinitely many $N$:
\begin{enumerate}
        \item $\phi$ is 0 outside of $\{h \in \GN : Wh+b \in \A \}$.
        \item $0 \leq \phi \leq \nu$.
        \item $\E(\phi | \FqN) \geq \delta$.
        \item $\|\phi\|_{\infty} \ll N$.
\end{enumerate}
\end{theorem}
\begin{remarks}
The introduction of $W$, known as the ``$W$-trick'', is quite common in this situation in arithmetic combinatorics, when we want to transfer results about dense sets to the primes. We will need the irreducible polynomials to be distributed sufficiently uniformly in congruence classes, and for this purpose, we will take $W$ to be a product of small irreducible polynomials. The value of $b$ is chosen
by the pigeonhole principle, so that the residue class of $b$ modulo
$W$ occupies a large proportion of $\A$.
\end{remarks}
\begin{proof}[Proof of Theorem \ref{FFGT} using Theorem \ref{FFESzemeredi} and Theorem \ref{FFGY}]
Suppose $N$ is such that the conclusions of Theorem \ref{FFGY}
holds. We partition $\GN$ into $q^{k}$ disjoint cylinders
$\mathcal{C}_{i}$ of radius $q^{N-k}$, so that
$|f_{1}-f_{2}|<q^{N-k}$ for any two polynomials $f_1, f_2$ in the
same cylinder. There must be a cylinder $\mathcal{C}_{i}$ on
which the average of $\phi$ is at least $\delta$. Let $\psi=\phi
1_{\mathcal{C}_{i}}$. Applying Theorem \ref{FFESzemeredi} to the
function $\psi$, we have $\E \left( \prod_{P \in \Gk} \psi(f+Pg)| f,
g \in \FqN \right) \geq c'(\frac{\delta}{q^{k}})-o(1)$. Because of
the bound on the magnitude of $\phi$, the contribution of the
products corresponding to trivial $k$-configurations is $o(1)$. Thus
for $N$ sufficiently large, $\psi$ is non-zero on some non trivial
$k$-configuration $\{ f+Pg| P \in \Gk  \} \subset \FqN$. A priori,
this is a $k$-configuration in $\FqN$. Because of the definition of
$\psi$, $f+Pg \in \mathcal{C}_{i}$ for every $P \in \Gk$. In
particular $q^{N-k}>|(f+g)-f|=|g|$, so that the above
$k$-configuration is indeed a $k$-configuration in $\GN$. Thus
$\{W(f+Pg)+b | P \in \Gk \}$ is a non-trivial $k$-configuration that
lies entirely in $\A$, since $\psi$ is supported in $\{h \in \GN :
Wh+b \in \A \}$.
\end{proof}

\begin{remarks}
The techniques here not only give infinitely many $k$-configurations, but also show that the number of such configurations is $\gg \frac{q^{2N}}{N^{K}}$, which is of correct magnitude in the context of the Hardy-Littlewood conjecture on tuples of primes. We remark that while more algebraic methods can generate configurations of irreducibles, e.g. the analog of the twin prime conjecture (\cite[Section 1.10]{pollack}), such methods don't give the correct bound (up to a constant).
\end{remarks}

The next sections are organized as follows. In Section \ref{S3} we establish Theorem \ref{FFSzemeredi}. In Section \ref{S4}, we define pseudorandom measures. Next, in Sections \ref{S5}, \ref{S6}, we introduce the Gowers norms, dual functions and their properties, which are necessary in our proof of the decomposition result in Section \ref{S7}. In Section \ref{S8}, we introduce arithmetic functions in $\Fq[t]$. We give in Section \ref{S9} the construction of a measure $\nu$ measure that majorizes the irreducible polynomials. Sections \ref{S10} and \ref{S11} will be devoted to establishing the pseudorandomness of $\nu$, thus finishing our proof of Theorem \ref{FFGT}.

\section{Szemer\'edi's theorem in function fields} \label{S3}

As aforementioned, we need an analog of Szemer\'edi's theorem in
$\Fq[t]$, namely that we can find non-trivial $k$-configurations
inside any subset of positive upper density of $\Fq[t]$:

\begin{proposition} \label{p1}
Given $\delta>0$. Then for $N$ sufficient large, $N \geq
N_0=N_0(q,k,\delta)$, in every subset $A$ of size $\delta q^{N}$ of
$\GN$, we can find polynomials $f, g \in \Fq[t], g \neq 0$ such that
$f+Pg \in A$ for every polynomial $P \in \Gk$.
\end{proposition}

There are at least two ways to see this. It is an immediate
consequence of a far more general result of
Bergelson-Leibman-McCutcheon \cite{blm}:

\begin{theorem}[Polynomial Szemer\'edi for countable integral domains]
Let $K$ be a countable integral domain, $M$ be a finitely generated
$K$-module, $p_1, \ldots, p_{n}$ be polynomials $K \rightarrow M$
such that $p_i(0)=0$ for every $i$. Then for any set $A \subset
M$ of upper Banach density $d^{*}(A)>0$, there exist $d \in K, d
\neq 0$ and $a \in M$ such that $a+p_{i}(d) \in A$ for every
$i=1,\ldots,n$.
\end{theorem}
When $K=M=\Fq[t], p_1, \ldots, p_{n}$ are the linear polynomials $g \mapsto Pg$, where $P \in \Gk$, then we have the desired result.

Proposition \ref{p1} can also be done by using the density
Hales-Jewett theorem. Before stating it, we need some definitions:

\begin{definition}
Given a set $A = \{ a_{1}, a_{2}, \ldots a_{t} \}$, the
$n$-dimension combinatorial space on $A$ is $C_{t}^{n} = \{ (x_1,
x_2, \ldots, x_{n}) : x_{i} \in A \textrm{ for } i=1,2, \ldots n \}$. A
combinatorial line in $C_{t}^{n}$ is a collection of $t$ points
$x^{(1)}, x^{(2)}, \ldots x^{(t)} \in C_{t}^{n}$ such that for some subset $I\neq \emptyset$ of $\{1,\ldots,n\}$ (the ``active'' coordinates), and fixed elements $(b_{i})_{i \in I}$, we have
$$x^{(j)}_{i}=
\left\{
  \begin{array}{ll}
    a_{j}, & \hbox{if $i \in I$ ;} \\
    b_{i}, & \hbox{if $i \not \in I$.}
  \end{array}
\right.$$
for any $j=1,\ldots,t$.
\end{definition}
We can think of a combinatorial space as the set of
all words of length $n$ on $t$ letters. For example, if $A=\{a,b,c\}$ then $\{axbx:x=a,b,c\}$ is a combinatorial line in the $4$-dimensional combinatorial space on $A$.
Note that a combinatorial line is a stronger notion than a geometric line in
that it``looks like" a line in every reordering of the underlying
set $A$.

The classical Hales-Jewett theorem \cite{hales-jewett}, \cite{grs} says

\begin{theorem}[Hales-Jewett]For every $r,t$, there is a number $HJ(r,t)$ such that if $n \geq HJ(r,t)$, if the points of the $n$-dimensional combinatorial space on $t$ elements are colored by $r$ colors, there
exists a monochromatic combinatorial line.
\end{theorem}

The density Hales-Jewett theorem says we can always locate such a
line in the most used color. It is to the classical Hales-Jewett
theorem like Szemer\'edi's theorem is to van der Waerden's theorem. The only known proof is due to Furstenberg and Katznelson \cite{fk} and uses ergodic theory.

\begin{theorem}[Density Hales-Jewett]For every $\delta>0,t \in \Z^{+}$, there is a number $HJ(\delta,t)$ such that if $n \geq HJ(r,t)$, in every subset of size $\delta t^{n}$ of the $n$-combinatorial space on $t$ elements,
there is a combinatorial line which lies entirely in this subset.
\end{theorem}

\begin{proof}[Proof of Proposition \ref{p1} using the density Hales-Jewett theorem]
We consider the set $X$ of all $k$-tuples $(a_0, a_1, \ldots,
a_{k-1})$ where $a_{i} \in \Fq$. Its cardinality is $K=q^{k}$. Let $N_0 = HJ(\delta, K)$, then $\mathbf{G}_{kN_{0}}$ can be identified with the combinatorial space of dimension $N_0$ over $X$, by identifying a polynomial with its $qN_{0}$ coefficients, divided into $N_{0}$ blocks of length $k$. By the definition of $N_{0}$, if we choose $\delta q^{N_{0}k}$ points out of $\mathbf{G}_{kN_{0}}$, there must exist a combinatorial line. It is easy to see that a combinatorial line in this space corresponds to a non-trivial $k$-configuration in $\Fq[t]$. Clearly if the statement is true $N=N_0$, then it is true for all $N \geq N_0$.
\end{proof}

A Varnavides argument shows that not only is there such a $k$-configuration, but there are in fact many of them:

\begin{proposition}\label{p2}
Given $\delta>0$. then there is a constant $c(\delta)>0$ such that,
for $N$ sufficient large, in every subset $A$ of size $\delta q^{N}$
of $\GN$, we can find at least $c(\delta) q^{2N}$ $k$-configurations.
\end{proposition}
\begin{proof}
Let $m=N_0(q,k,\frac{\delta}{2})$, where $N_0$ is the function in Proposition \ref{p1}, and suppose $N\geq m$. We claim that among $q^{2N-m}$ $m$-configurations in $\GN$, there are at least $\geq \frac{\delta}{2}q^{2N-2m}$ of them on which the density of $A$ is at least $\frac{\delta}{2}$.

Indeed, let us count the number of pairs $(V,h)$ where $V$ is a $m$-configuration in $\GN$ and $h \in A \cap V$. On the one hand, since for every given point in $\GN$ there are $q^{N-m}$ $m$-configurations in $\GN$ containing it (why?), the number of such pairs is $\delta q^{N}q^{N-m}=\delta q^{2N-m}$. Each $k$-configuration $V$ on which the density of $A$ is at most $\delta/2$ contributes at most $\frac{\delta}{2} q^{m}q^{2N-2m}=\frac{\delta}{2}q^{2N-m}$ pairs. Thus the contribution of those $V$ on which the density of $A$ is at least $\delta/2$ is at least $\delta q^{2N-m}-\frac{\delta}{2} q^{2N-m}=\frac{\delta}{2}q^{2N-m}$. Therefore, the number of $m$-configuration $V$ on which the density of $A$ is $\geq \frac{\delta}{2}$ is at least $\frac{\delta}{2}q^{2N-2m}$.

From the definition of $m$, each such $m$-configuration (with the exception of at most $q^{N}$ trivial $m$-configurations) contains a non-trivial $k$-configuration. The number of times a given $k$-configuration is counted is at most $q^{m-k}$ (why?). Thus the number of non-trivial $k$-configurations in $\GN$ is at least $q^{k-m}(\frac{\delta}{2}q^{2N-m}-q^{N}) \gg_{\delta} q^{2N}$, as desired.
\end{proof}

From this Theorem \ref{FFSzemeredi} easily follows.
\begin{proof}[Proof of Theorem \ref{FFSzemeredi}]
Suppose $\E(\phi | \GN) \geq \delta$. Let $B\subset\GN$ be the set on which $\phi$ is $\geq \delta/2$. Then $|B|+\frac{\delta}{2}(q^{N}-|B|) \geq \delta q^{N}$, so that $|B| \geq \frac{\delta}{2}q^{N}$. Proposition \ref{p2} implies that $B$ contains at least $c(\frac{\delta}{2})q^{2N}$ $k$-configurations. Thus $\E \left(\prod_{P \in \Gk} \phi(f+Pg)| f, g \in \FqN \right) \geq (\frac{\delta}{2})^{k}c(\frac{\delta}{2})$.
\end{proof}

\begin{remark}
In contrast with the usual Szemer\'edi theorem for the integers,
where we have a quantitative proof due to Gowers \cite{gowers-4},
\cite{gowers}, the proofs of the Bergelson-Leibman-McCutcheon
theorem uses ergodic theory and
therefore do not give any bound of $c(\delta)$ in terms of $\delta$. As for the density Hales-Jewett Theorem, As for the density Hales-Jewett theorem, just until very recently, Gowers et. al. \cite{polymath} announced to have found combinatorial proofs, from which some bounds might be extracted, but still far weaker than what is known for the integers. Thus we don't seek to find a bound for the first occurrence in the
irreducible polynomials of the configurations $\{f+Pg| P \in \Gk\}$.
\end{remark}

\section{Pseudorandom measures} \label{S4}

A measure is a function\footnote{More precisely, it is a family $\{\nu_{N}\}_{N \in \Z^{+}}$ such that for each $N$,  $\nu_{N}$ is a function from  $\FqN \rightarrow \R$.} $\nu : \FqN \rightarrow [0, \infty)$. Strictly speaking, $\nu \mu$ should be called a measure rather than $\nu$, where $\mu$ is the normalized counting measure on $\GN$ (and therefore we should think of $\nu$ as the Radon-Nikodym derivative of a random measure with respect to the normalized counting measure). However, we keep this naming to be consistent with \cite{gt-primes}. A pseudorandom measure is a measure satisfying the two conditions defined below:

\begin{definition}[Linear forms condition]We say that a measure $\nu: \FqN \rightarrow \infty$ satisfies the $(m_0,n_0,k_0)$-linear forms condition if whenever we have $m\leq m_0$ linear forms in $n\leq n_0$ variables $\psi_1,\ldots, \psi_{m}:(\FqN)^{n} \rightarrow \GN$ of the form
$$\psi_{i}(\mathbf{f})=\sum_{j=1}^{n}L_{ij}f_{i}+b_{i}$$
such that all the coefficients $L_{ij}$ are in the set\footnote{Note that this set can be embedded in to $\FqN$ for every $N \geq k_{0}$. Of course, the embedding depends on our choice of the irreducible polynomial $f_{N}$ underlying $\FqN$.} $\{ \frac{f}{g}|f,g \in \mathbf{G}_{k_0} \}$, and no two of the vectors $(L_{ij})_{1\leq j \leq n}, i=1,\ldots,m$, are proportional, then we have
\begin{equation} \label{linear forms}
 \E \left( \nu( \psi_1( \mathbf{f}) ) \ldots \nu ( \psi_{m}(\mathbf{f}) ) | \mathbf{f}\in \left( \FqN \right)^{n} \right)  = 1+o_{N \rightarrow \infty}(1)
\end{equation}
\end{definition}
In particular $\nu$ satisfies the linear forms condition then $\E (\nu (f) | f \in \GN) = 1+o(1)$.
Note that we require the $o_{N \rightarrow \infty}(1)$ to be uniform in all choices of $b_1, \ldots, b_{m} \in\FqN$.

\begin{definition}[Correlation condition] We say that a measure $\nu: \FqN \rightarrow \infty$ satisfies the $l_0$-correlation condition if whenever we have $l \leq l_0$ linear forms of the form $f+h_1, \ldots f+h_{l}$ with $h_1, h_2, \ldots h_{l} \in \FqN$ , then
\begin{equation} \label{correlation}
\E \left( \nu(f+h_1) \ldots \nu(f+h_{l}) | f \in \FqN \right) \leq \sum_{1 \leq i \leq j \leq q } \tau(h_{i}-h_{j})
\end{equation}
where $\tau$ is a function $\GN \rightarrow \R^{+}$ having the property that $\E(\tau(f)^p | f\in \FqN) =O_{p}(1)$ for every $p>1$.
\end{definition}
The point is that the function $\tau$ is not necessarily bounded as $N$ tends to infinity, but its $L^{p}$-norm is always bounded.

\begin{definition}[Pseudorandom measures] A measure $\nu: \FqN \rightarrow \infty$ is called $k$-pseudorandom if it satisfies the $(K2^{K-1},3K-4,k)$-linear forms condition and the $2^{K-1}$-correlation condition (recall that $K=q^{k}$).
\end{definition}
\begin{remark}
The exact values of the parameters $m_0,l_0,k_0,l_0$ are not important, since we will see that for any $(m_0,l_0,k_0,l_0)$, we can find a measure that satisfies the $(m_0,l_0,k_0)$-linear forms condition and the $l_0$-correlation condition. However, it is essential in the construction that these values are finite.
\end{remark}
From now on we will refer to $k$-pseudorandom measures as pseudorandom measures.
\begin{lemma}\label{nu+1}
If $\nu$ is pseudorandom, then so is $\nu_{1/2}=(\nu+1)/2$. More generally, for any $0<\alpha<1$, $\nu_{\alpha}=(1-\alpha)\nu+\alpha$ is also pseudorandom.
\end{lemma}
In practice we will be dealing with $\nu_{1/2}$ and $\nu_{1/4}$.
\begin{proof}
Suppose we want to verify the condition (\ref{linear forms}) for $\nu+1$. If we expand the product of $m$ factors $(\nu(\psi_{i})+1)$, we will have $2^{m}$ terms, each of them is a product of $m$ or fewer factors of the form $\nu(\psi_{i})$. By the linear forms condition for $\nu$ we know that each such term is $1+o(1)$, thus the linear forms condition for $\nu_{1/2}$ follows (with possibly different $o(1)$ term). The correlation condition is checked similarly. The proof for $\nu_{\alpha}$ is similar.
\end{proof}
\section{Gowers norms} \label{S5}
One efficient tool in counting linear patterns is Gowers norms. The Gowers norms are first used by Gowers in his proof of Szemer\'edi's theorem \cite{gowers-4}, \cite{gowers}. It has a parallel counterpart in ergodic theory, known as the Host-Kra seminorm \cite{hk}.
\begin{definition}[Gowers norm] Let $G$ be a finite abelian group, $\phi$ a complex-valued function
defined on $G$. For $\mathbf{\omega} = (\omega_1, \ldots \omega_{d})
\in \{0,1 \}^{d}$, let $|\mathbf{\omega}|=\omega_1 + \ldots +
\omega_{d}$. Also, let $C$ be the complex conjugation. We define
the $d$-th Gowers norm of $\phi$ to be
$$\| \phi \|_{U^{d}(G)}=\left(\E_{x,h_1,\ldots,h_{k} \in G} \prod_{\mathbf{\omega} \in \{0,1\}^{d}} C^{|\mathbf{\omega}|}\phi(x+\omega_1 h_1+ \ldots +\omega_{d} h_{d})
\right)^{1/2^{d}}.$$
\end{definition}
Alternatively, the Gowers norms $\| \cdot \|_{U^{d}(G)}$ can be defined recursively as follows:
$$\| \phi \|_{U^{1}(G)}=|\E(\phi|G)|$$
$$\| \phi \|_{U^{d+1}(G)}^{2^{d+1}}=\E \left( \| \phi \cdot \phi_{t} \|_{U^{d}(G)}^{2^{d}}| t \in G \right)$$
where $\phi_{t}$ is the function $\phi_{t}(x)=\phi(t+x)$.

The following facts about the Gowers norms are standard and the proofs can be found in \cite{gt-primes} or \cite{gt-inverse-u3}

\begin{proposition}[Gowers-Cauchy-Schwarz inequality]
Suppose $\phi_{\mathbf{\omega}}$, for $\mathbf{\omega} \in \{0,1\}^{d}$, are $2^{d}$ functions: $G \rightarrow \C$. Then
$$\E \left( \prod_{\mathbf{\omega} \in \{0,1\}^{d}} \phi_{\mathbf{\omega}}(x+\omega_1 h_1 + \cdots +\omega_{d} h_{d})\right) \leq \prod_{\mathbf{\omega} \in \{0,1\}^{d}}\| \phi_{\mathbf{\omega}} \|_{U^{d}(G)} $$
\end{proposition}

\begin{proposition}
For every $\phi$, the sequence $\| \phi \|_{U^{d}(G)}, d=1,2,\ldots $ is an increasing sequence. In particular for every $d \geq 1,\| \phi \|_{U^{d}(G)}\geq \| \phi \|_{U^{1}(G)}=| \E(\phi|G)|$.
\end{proposition}

\begin{proposition}
For every $d \geq 2, \| \cdot \|_{U^{d}(G)}$ is indeed a norm on $\C^{G}$, the space of complex functions on $G$.
\end{proposition}

Henceforth, if the context is clear, we will assume $G=\FqN$ and omit the group in the notation of the Gowers norm. In practice we will be dealing with the $U^{K-1}$ norm, where $K=q^{k}$.

Our first observation is that a pseudorandom measure is close to the uniform measure in the $U^{K-1}$ norm.

\begin{lemma}\label{n}
Let $\nu$ be a pseudorandom measure on $\FqN$. Then $\|\nu-1\|_{U^{K-1}}=o(1)$. Consequently, $\|\nu\|_{U^{K-1}}=1+o(1)$.
\end{lemma}
\begin{proof}The proof is exactly the same as in Lemma 5.2 in \cite{gt-primes}, so we reproduce it briefly here for the case $K=3$, i.e. $\|\nu-1\|_{U^2}=o(1)$. Equivalently, we have to show
$$\|\nu-1\|_{U^2}^{4}=\E \left( (\nu(f)-1)(\nu(f+h_1)-1)(\nu(f+h_2)-1)(\nu(f+h_1+h_2)-1) \Big| f,h_1,h_2 \in \FqN \right)=o(1)$$
If we expand the expectation, we will have a sum of 16 terms, each term is the expectation of a product of $\nu$ composed with linear forms. By the linear forms condition, each term is $1+o(1)$. Thus our expression is $\sum_{I \subset \{1,2\}}(-1)^{|I|}(1+o(1))=1+o(1)$, as required.
\end{proof}
As mentioned before, Gowers norms are effective in counting linear patterns, as witnessed by the following

\begin{proposition}[Generalized von Neumann\footnote{Green and Tao call this type of inequalities generalized von Neumann theorems to emphasize their connection with the classical von Neumann theorem in ergodic theory.}] \label{gvN}
Suppose $\nu$ is pseudorandom. Let $(\phi_{P})_{P\in \Gk}$ be functions bounded in absolute value by $\nu$. Then
$$\E \left( \prod_{P \in \Gk}\phi_{P}(f+Pg) \Big| f,g \in \FqN \right) \leq \min_{P} \| \phi_{P} \|_{U^{K-1}}+o_{N \rightarrow \infty}(1)$$
\end{proposition}
Before proving this proposition let us derive the following
\begin{corollary}\label{gvN2}
If the $\phi_{P}$ are bounded by $3+\nu$ then $\E \left( \prod_{P \in \Gk}\phi_{P}(f+Pg) \Big| f,g \in \FqN \right) \leq 4^{K} \min_{P} \| \phi_{P} \|_{U^{K-1}}+o_{N \rightarrow \infty}(1)$.
\end{corollary}
\begin{proof}Just divide each $\phi_{P}$ by 4, and apply Proposition \ref{gvN} for functions bounded by the measure $\nu_{1/4}=(\nu+3)/4$, which is pseudorandom according to Lemma \ref{nu+1}.
\end{proof}

\begin{proof}[Proof of Proposition \ref{gvN}]
The proof is exactly the same as in Proposition 5.3 in \cite{gt-primes}, so we just reproduce it briefly here. Since $\Gk$ is a subgroup of $(\FqN,+)$, by a change of variable if need be, it suffices to show that $\E \left( \prod_{P \in \Gk}\phi_{P}(f+Pg))\Big| f,g \in \GN \right) \leq \| \phi_{0} \|_{U^{K-1}}+o_{N \rightarrow \infty}(1)$.

We claim something slightly more general, namely that for any $m$ distinct non-zero polynomials $P_1, P_2, \ldots, P_m \in \Gk$, we have
$$\E \left( \phi_0(f)\phi_1(f+P_1g)\cdots \phi_{m}(f+P_{m}g) \Big| f,g \in \FqN \right) \leq \| \phi_{0} \|_{U^{m}}+o_{N \rightarrow \infty}(1)$$
Proposition \ref{gvN} is a special case of this when $m=K-1$.
For $\mathbf{y}=(y_1,\ldots, y_{m}) \in (\FqN)^{m}$ we define the linear forms
$$ \Psi_0(\mathbf{y})=y_1+y_2+\cdots+y_{m}$$
$$ \Psi_{i}(\mathbf{y})=\sum_{j=1, j \neq i}^{m} \left(1-\frac{P_{i}}{P_{j}} \right) y_{j}$$

Thus for every $i=1,\ldots,m$, $\Psi_{i}(y_1,\ldots, y_{m})$ does not depend on $y_{i}$, which will be crucial in our later use of Cauchy-Schwarz.
If $f=\sum_{i=1}^{m}y_{i}, g=-\sum_{i=1}^{m}\frac{y_{i}}{P_{i}}$, then $f+P_{i}g=\Psi_{i}(\mathbf{y})$.
It is easy to see that the map $\mathbf{y} \mapsto (f,g)$ is $q^{N(m-2)}$-to-one, so that
$$ \E \left( \phi_0(f)\phi_1(f+P_1g) \cdots \phi_{m}(f+P_{m}g) \Big| f,g \in \FqN \right)= \E \left( \prod_{i=0}^{m} \phi_{i}(\Psi_{i}(\mathbf{y}))\Big|\mathbf{y} \in (\FqN)^{m} \right) $$
For simplicity let us treat the case $m=2$. The argument then extends straighforwardly to the general case. Let's call the left hand side $J_0$.
Since $\phi_2$ is bounded by $\nu$, we have:
$$J_0  \leq  \E \left( \nu(\Psi_2(y_1)) \left| \E \left( \phi_1(\Psi_1(y_2)) \phi_0(\Psi_0(y_1,y_2)) \Big|  y_2 \in \FqN \right) \right| \Big| y_1 \in \FqN \right) $$

By Cauchy-Schwarz,
\begin{eqnarray}
J_0^2 & \leq & \E \left(  \nu(\Psi_2(y_1)) \Big| y_1 \in \FqN \right) \E \left( \nu(\Psi_2(y_1)) \E \left( \phi_1(\Psi_1(y_1,y_2)) \phi_0(\Psi_0(y_1,y_2)) \Big| y_2 \in \FqN \right)^2   \Big| y_1 \in \FqN \right) \nonumber \\
&\leq& (1+o(1)) J_1 \nonumber
\end{eqnarray}
where we have used the linear forms condition for $\Psi_2(y_1)$, and
\begin{eqnarray}
J_1 &=& \E \left( \E \left( \phi_1(\Psi_1(y_1,y_2)) \phi_0(\Psi_0(y_1,y_2)) \Big| y_2 \in \FqN \right)^2 \nu(\Psi_2(y_1)) \Big| y_1 \in \FqN \right) \nonumber \\
&=& \E \left( \phi_1(\Psi_1(y_2)) \phi_0(\Psi_0(y_1,y_2))\phi_1(\Psi_1(y_2')) \phi_0(\Psi_0(y_1,y_2')) \nu(\Psi_2(y_1)) \Big| y_1,y_2,y_2' \in \FqN \right)\nonumber \\
&=& \E \left( \phi_1(\Psi_1(y_2)) \phi_1(\Psi_1(y_2')) \left|  \E \left( \phi_0(\Psi_0(y_1,y_2))\phi_0(\Psi_0(y_1,y_2')) \nu(\Psi_2(y_1)) \Big| y_1 \in \FqN  \right) \right| \Big| y_2,y_2' \in \FqN \right) \nonumber \\
&=& \E \left( \nu(\Psi_1(y_2)) \nu(\Psi_1(y_2')) \left|  \E \left( \phi_0(\Psi_0(y_1,y_2))\phi_0(\Psi_0(y_1,y_2')) \nu(\Psi_2(y_1)) \Big| y_1 \in \FqN  \right) \right| \Big| y_2,y_2' \in \FqN \right) \nonumber
\end{eqnarray}
Note that we have eliminated $\phi_2$ in $J_1$. Next, again by Cauchy-Schwarz,
\begin{eqnarray}
J_1^2 &=& \E \left( \nu(\Psi_1(y_2))\nu(\Psi_1(y_2')) \Big| y_2,y_2' \in \FqN \right) \times \nonumber \\
& &\times \E \left( \E \left( \phi_0(\Psi_0(y_1,y_2))\phi_0(\Psi_0(y_1,y_2')) \nu(\Psi_2(y_1)) \Big| y_1 \in \FqN \right)^2 \nu(\Psi_1(y_2))\nu(\Psi_1(y_2')) \Big| y_2, y_2' \in \FqN \right) \nonumber \\
& \leq & (1+o(1) J_2 \nonumber
\end{eqnarray}
where we have used the linear forms condition for the forms $\Psi_1(y_2)$ and $\Psi_1(y_2')$, and
\begin{eqnarray}
J_2 &=& \E \left( \E \left( \phi_0(\Psi_0(y_1,y_2))\phi_0(\Psi_0(y_1,y_2')) \nu(\Psi_2(y_1)) \Big| y_1 \in \FqN \right)^2 \nu(\Psi_1(y_2))\nu(\Psi_1(y_2')) \Big| y_2, y_2' \in \FqN \right) \nonumber \\
&=& \E \Big( \phi_0(\Psi_0(y_1,y_2))\phi_0(\Psi_0(y_1,y_2')\phi_0(\Psi_0(y_1',y_2))\phi_0(\Psi_0(y_1',y_2')) \times \nonumber \\
& & \times \nu(\Psi_2(y_1))\nu(\Psi_2(y_1'))\nu(\Psi_1(y_2))\nu(\Psi_1(y_2')) \Big| y_1,y_1',y_2,y_2' \in \FqN \Big) \nonumber
\end{eqnarray}
Note that we have eliminated $\phi_1$ in $J_2$. Recall that $\Psi_0(y_1,y_2)=y_1+y_2$. Let us re-parameterize the cube $\{\Psi_0(y_1,y_2),\Psi_0(y_1,y_2'),\Psi_0(y_1',y_2), \Psi_0(y_1',y_2') \}=\{f,f+h_1,f+h_2,f+h_1+h_2\}$. Then
\begin{eqnarray}
J_2 &=& \E \Big( \phi_0(f)\phi_0(f+h_1)\phi_0(f+h_2)\phi_0(f+h_1+h_2) \times \nonumber \\
& & \times \nu(\Psi_1(f)) \nu(\Psi_1(f+h_1)) \nu(\Psi_2(f+h_2))\nu(\Psi_2(f+h_1+h_2)) \Big| f,h_1,h_2 \in \FqN \Big) \nonumber
\end{eqnarray}
If it was not for the factor $W(f,h_1,h_2)=\nu(\Psi_1(f)) \nu(\Psi_1(f+h_1)) \nu(\Psi_2(f+h_2))\nu(\Psi_2(f+h_1+h_2))$, then $J_2$ would be equal to $\| \phi_0 \|_{U^2}^4$. We have to show that $J_2=\| \phi_0 \|_{U^2}^4+o(1)$.

Indeed,
\begin{eqnarray}
J_2-\| \phi_0 \|_{U^2}^4 &=& \E \left( (W(f,h_1,h_2)-1)\phi_0(f)\phi_0(f+h_1)\phi_0(f+h_2)\phi_0(f+h_1+h_2) \Big| f,h_1,h_2 \in \FqN \right) \nonumber \\
& \leq & \E \left( (W(f,h_1,h_2)-1)\nu(f)\nu(f+h_1)\nu(f+h_2)\nu(f+h_1+h_2) \Big| f,h_1,h_2 \in \FqN \right) \nonumber
\end{eqnarray}
By Cauchy-Schwarz,
\begin{eqnarray}
\left| J_2-\| \phi_0 \|_{U^2}^4 \right|^2 &\leq& \E \left(\nu(f)\nu(f+h_1)\nu(f+h_2)\nu(f+h_1+h_2) \Big| f,h_1,h_2 \in \FqN \right) \times \nonumber \\
& & \times \E \left( (W(f,h_1,h_2)-1)^2 \nu(f)\nu(f+h_1)\nu(f+h_2)\nu(f+h_1+h_2) \Big| f,h_1,h_2 \in \FqN \right) \nonumber
\end{eqnarray}
Thus it suffices to show the following two claims:
\begin{claim}
$\E \left(\nu(f)\nu(f+h_1)\nu(f+h_2)\nu(f+h_1+h_2) \Big| f,h_1,h_2 \in \FqN \right)=1+o(1)$
\end{claim}
\begin{claim}
$\E \left( (W(f,h_1,h_2)-1)^2 \nu(f)\nu(f+h_1)\nu(f+h_2)\nu(f+h_1+h_2) \Big| f,h_1,h_2 \in \FqN \right)=o(1)$
\end{claim}
The first claim follows from the linear forms condition for 4 forms. As for the second claim, we expand the left hand side as
\begin{eqnarray}
& & \E \left( (W(f,h_1,h_2)^2 \nu(f)\nu(f+h_1)\nu(f+h_2)\nu(f+h_1+h_2) \Big| f,h_1,h_2 \in \FqN \right) \nonumber \\
&-& 2\E \left( (W(f,h_1,h_2) \nu(f)\nu(f+h_1)\nu(f+h_2)\nu(f+h_1+h_2) \Big| f,h_1,h_2 \in \FqN \right) \nonumber \\
&+& \E \left( \nu(f)\nu(f+h_1)\nu(f+h_2)\nu(f+h_1+h_2) \Big| f,h_1,h_2 \in \FqN \right) \nonumber
\end{eqnarray}
Using the linear forms condition for 12,8 and 4 forms respectively, we see that this is $(1+o(1))-2(1+o(1))+(1+o(1))=o(1)$, as required.
\end{proof}
\begin{remark}
In the general case, we will need the linear forms condition for $K2^{K-1}$ linear forms in $3K-4$ variables.
\end{remark}
\begin{remark}Our use of this generalized von Neumann inequality follows Green-Tao and thus is genuinely different from Gowers'. In his proof of Szemer\'edi's theorem, Gowers used the following fact, which is now known as a weak form of the the Gowers Inverse Conjecture \cite{gt-inverse-u3}: If the Gowers norm of a function is large, then it must correlate locally with a polynomial phase. Green and Tao used this inequality for their transference principle, namely to transfer Szemer\'edi's theorem from the uniform measure to pseudorandom measures.
\end{remark}

\section{Gowers anti-uniformity} \label{S6}

\begin{definition}
For a real function $\phi$ on $\FqN$, define its $U^{d}$ dual function $\mathcal{D}_{d}\phi$ by
$$\mathcal{D}_{d}\phi= \E \left( \prod_{\mathbf{\omega} \in \{0,1\}^{d}, \omega \neq 0} \phi(x+\omega_1 h_1+ \ldots +\omega_{d} h_{d})\Big|x,h_1,\ldots,h_{d} \in \FqN \right)  $$
\end{definition}

From the definitions of the Gowers $U^{d}$ norm and $U^{d}$ dual functions it follows immediately that
\begin{lemma} \label{property1}
$\langle \phi, \mathcal{D}_{d}\phi \rangle = \ \| \phi \|_{U^{d}(G)}^{2^{d}}$
\end{lemma}

From a functional analytic point of view $\mathcal{D}_{d}\phi$ may be regarded as a ``support functional'' of $\phi$, with the difference that $\Dphi$ is not linear. From now on we will be working with the $U^{K-1}$ dual functions
We will be particularly interested in the dual functions of functions bounded by a pseudorandom measure $\nu$.

\begin{lemma} \label{property2}
If $0 \leq \phi \leq \nu$, then $\langle \phi, \mathcal{D}_{K-1}\phi \rangle = 1+o(1)$.
\end{lemma}
\begin{proof}This follows from Lemmas \ref{n} and \ref{property1}.
\end{proof}

\begin{lemma} \label{property3}
For every $m$ there is a constant $C(m)$ such that if $0 \leq \phi_1, \ldots, \phi_{m} \leq \nu$, then
$\|\mathcal{D}_{K-1}\phi_1 \cdots \mathcal{D}_{K-1} \phi_{m} \|_{U^{K-1}}^{*} \leq C(m)$, where $\| \cdot \|_{U^{K-1}}^{*}$ is the dual norm of $\| \cdot \|_{U^{K-1}}$ (defined in the usual way $\| f \|_{U^{K-1}}^{*} = \sup\{ |\langle f,g \rangle| : \| g \|_{U^{K-1}} \leq 1 \}$).
\end{lemma}

This is by far the most important property of the dual functions, and perhaps surprising, since $m$ is not bounded, while the number of forms in the linear forms condition and correlation condition is bounded. However, this comes from the fact that the exponent $p$ of the function $\tau$ in the correlation condition is not bounded. In Reingold-Trevisan-Tulsiani-Vadhan's language this means that $\nu$ is indistinguishable to the uniform measure according to the family $\{ \mathcal{D}_{K-1}\phi_1 \cdots \mathcal{D}_{K-1} \phi_{m}: 0 \leq \phi_{i}\leq \nu \}$.
\begin{proof} The proof is exactly the same as in Lemma 6.3 in \cite{gt-primes}, so we will reproduce it here for the case $K=3,m=2$.
It suffices to show that for any function $\psi$ with $\| \psi \|_{U^{2}}\leq 1$, we have $\langle \psi, \Dphi_1 \Dphi_2 \rangle = O(1)$. We write out this as
\begin{eqnarray}
 & & \E \Big( \psi(f) \E \left( \phi_1(f+h_1)\phi_1(f+h_2)\phi_1(f+h_1+h_2) \Big| h_1, h_2 \in \FqN \right) \times \nonumber \\
 & & \times \E \left( \phi_2(f+k_1)\phi_2(f+k_2)\phi_2(f+k_1+k_2) \Big| k_1, k_2 \in \FqN \right) \Big| f \in \FqN \Big) \nonumber \\
 &=& \E \Big( \psi(f) \E \left( \phi_1(f+h_1+g_1)\phi_1(f+h_2+g_2)\phi_1(f+h_1+g_1+h_2+g_2) \Big| h_1, h_2 \in \FqN \right) \times \nonumber \\
 & & \times \E \left( \phi_2(f+k_1+g_1)\phi_2(f+k_2+g_2)\phi_2(f+k_1+g_1+k_2+g_2) \Big| k_1, k_2 \in \FqN \right) \Big| f,g_1,g_2 \in \FqN \Big) \nonumber
\end{eqnarray}
We rewrite this as
\begin{eqnarray}
& & \E \Big( \E \Big( \psi(f) \phi_1(f+g_1+h_1) \phi_2(f+g_1+k_1) \phi_1(f+g_2+h_2) \phi_2(f+g_2+k_2) \times \nonumber \\
& & \times \phi_1(f+g_2+h_1+h_2) \phi_2(f+g_2+k_1+k_2) \Big| f, g_1, g_2 \in \FqN \Big) \Big| h_1,h_2,k_1,k_2 \in \FqN \Big) \nonumber
\end{eqnarray}
By the Gowers-Cauchy-Schwarz inequality this is at most
\begin{eqnarray}
& & \E \Big( \| \psi \|_{U^2} \Big\| \phi_1(\cdot+h_1)\phi_2(\cdot+k_1) \Big\|_{U^2} \Big\| \phi_1(\cdot+h_2)\phi_2(\cdot+k_2) \Big\|_{U^2} \times \nonumber \\
& & \times \Big\| \phi_1(\cdot+h_1+h_2)\phi_2(\cdot+k_1+k_2) \Big\|_{U^2} \Big| h_1,h_2,k_1,k_2 \in \FqN \Big) \nonumber
\end{eqnarray}
By H\"{o}lder's inequality, and since $\|\psi\|_{U^2} \leq 1$, this is at most
$$ \E \left(\|\phi_1(\cdot+h) \phi_2(\cdot+k) \|_{U^2}^3 \Big| h,k \in \FqN \right)^{1/3} \leq \E \left(\|\phi_1(\cdot+h) \phi_2(\cdot+k) \|_{U^2}^4 \Big| h,k \in \FqN \right)^{1/4}$$
Thus it suffices to show $\E \left(\|\phi_1(\cdot+h) \phi_2(\cdot+k) \|_{U^2}^4 \Big| h,k \in \FqN \right)=O(1)$. If we expand this out then it is equal to
\begin{eqnarray}
& & \E \Big( \E \Big( \phi_1(f+h+g_1) \phi_2(f+k+g_1) \phi_1(f+h+g_2) \phi_2(f+k+g_2) \times \nonumber \\
& & \times \phi_1(f+h+g_1+g_2) \phi_2(f+k+g_1+g_2) \Big| f,g_1,g_2 \in \FqN \Big)  \Big| h,k \in \FqN \Big) \nonumber
\end{eqnarray}
If we interchange the order of summation then this is equal to
\begin{eqnarray}
& & \E \Big( \E \Big( \phi_1(f+h+g_1) \phi_2(f+k+g_1) \phi_1(f+h+g_2) \phi_2(f+k+g_2) \times \nonumber \\
& & \times \phi_1(f+h+g_1+g_2) \phi_2(f+k+g_1+g_2) \Big| f,h,k \in \FqN \Big)  \Big| g_1,g_2 \in \FqN \Big) \nonumber \\
&=& \E \Big( \E \Big( \phi_1(h+g_1) \phi_2(k+g_1) \phi_1(h+g_2) \phi_2(k+g_2) \times \nonumber \\
& & \times \phi_1(h+g_1+g_2) \phi_2(k+g_1+g_2) \Big| h,k \in \FqN \Big)  \Big| g_1,g_2 \in \FqN \Big) \nonumber \\
&\leq& \E \left( \E \left( \nu(f)\nu(f+g_1)\nu(f+g_2)\nu(f+g_1+g_2) \Big| f\in \FqN \right)^2 \Big| g_1,g_2 \in \FqN \right) \nonumber
\end{eqnarray}
According to the correlation condition, and the triangle inequality, this is at most
\begin{eqnarray}
& & \E \left( \left( \tau(g_1)+\tau(g_2)+\tau(g_1-g_2)+\tau(g_1+g_2) \right)^2 \Big| g_1,g_2 \in \FqN \right) \nonumber \\
&\leq& \Big( \E(\tau(g_1)^2|g_1,g_2 \in \FqN)^{1/2}+\E(\tau(g_2)^2|g_1,g_2 \in \FqN)^{1/2}+ \nonumber \\
& & +\E(\tau(g_1-g_2)^2|g_1,g_2 \in \FqN)^{1/2}+\E(\tau(g_1+g_2)^2|g_1,g_2 \in \FqN)^{1/2} \Big)^2 \nonumber \\
&=& 4\E(\tau(g)^2|g\in \FqN) = O(1) \nonumber
\end{eqnarray}
as required.
\end{proof}
\begin{remark}
In the general case, we will need the correlation condition for $2^{K-1}$ forms.
\end{remark}
\section{A decomposition and a transference principle} \label{S7}
In this section we reproduce Gowers' proof \cite{gowers-gt} of the
Green-Tao-Ziegler theorem and use the latter to derive Theorem
\ref{FFESzemeredi}. The reader is nevertheless recommended for a
reading of the original paper for a survey about the interplay
between decomposition results and the use of the Hahn-Banach theorem
in arithmetic combinatorics.

We first forget for a moment the definitions of Gowers norms and dual functions, but instead axiomatize their properties as proved in Lemmas \ref{property1}, \ref{property2}, and \ref{property3}. Consider a finite set $G$ and let $\R^{G}$ be the set of all real functions on $G$ with the inner product $\langle f,g \rangle=\E_{x \in G}f(x)g(x)$.

\begin{definition}\label{QAP}
We say that a norm $\|\cdot\|$ on $\R^{G}$ is a quasi-algebra predual norm with respect to a convex, compact set $\F \subset \R^{G}$ if there is a function $c: \R^{+} \rightarrow R^{+}$, a function $C: \Z^{+} \rightarrow \R^{+}$, and an operator $\mathcal{D}: \R^{G} \rightarrow \R^{G}$ such that the following hold:
\begin{enumerate}
    \item $\langle f, \Df \rangle \leq 1$ for every $f \in \F$.
    \item $\langle f, \Df \rangle \geq c(\epsilon)$ for every $f \in \F$ with $\| f \| \geq \epsilon$.
    \item $\|\Df_1 \ldots \Df_{m} \|^{*} \leq C(m)$, where $\|\cdot\|^{*}$ is the dual norm of $\|\cdot\|$.
    \item The set $\{ \Df, f \in \F \}$ is compact and spans $\R^{G}$.
\end{enumerate}
\end{definition}
The reason why $\|\cdot\|$ is called a quasi-algebra predual norm is that the dual norm $\|\cdot\|^{*}$ is ``close'' to being an algebra norm (this will be made precise in Lemma \ref{BAC}). The application we have in mind is when $G=\FqN,\| \cdot \|$ is the (normalized) $U^{K-1}$ Gowers norm, the $\Df$ are the (normalized) $U^{K-1}$ dual functions, $\F$ is the space of nonnegative functions bounded by a pseudorandom measure $\nu$.

Associated to the norm $\|\cdot\|$, we will also consider the norm $\|g\|_{BAC}=\max \{ | \langle g,\Df \rangle |:f \in \F \}$ and its dual $\| \cdot \|_{BAC}^{*}$ (Since the set $\{ \Df, f \in \F \}$ is compact, $\|\cdot\|_{BAC}$ is indeed a norm). Here BAC stands for Basic Anti-uniform Correlation. Thus $\|\cdot \|$ and $\| \cdot \|_{BAC}$ are equivalent in a sense that if $f \in \F$ and $\|f \| \geq \epsilon$ then $\| f \|_{BAC}\geq c(\epsilon)$.

The following gives a simple characterization of the $\|\cdot\|_{BAC}^{*}$ norm.

\begin{lemma}
$\| f \|_{BAC}^{*}= \inf \{\sum_{i=1}^{k} |\lambda_{i}|: f=
\sum_{i=1}^{k} \lambda_{i}\Df_{i}, f_1,\ldots, f_{k}\in \F \}$.
\end{lemma}

\begin{proof}
This can be proven using Farkas' lemma \cite[Section 1.16]{tao-blog} (which is another incarnation
of the Hahn-Banach theorem). We can also do this in a relatively
simpler way as follows: define the norm $\|f\|_{0}= \inf
\{\sum_{i=1}^{k} |\lambda_{i}|: f= \sum_{i=1}^{k}
\lambda_{i}\Df_{i}, f_1,\ldots, f_{k}\in \F \}$ (which exists by our
assumption that the $\Df, f \in \F$ span $\R^{G}$), we have to show
that the dual norm $\|\cdot\|_{0}^{*}$ is equal to $\|\cdot\|_{BAC}$
(note that here we are using Hahn-Banach implicitly!).

Suppose $f, g \in \R^{G}$. For any decomposition $f= \sum_{i=1}^{k}
\lambda_{i}\Df_{i}, f_1,\ldots, f_{k}\in \F$, we have $|\langle g,f
\rangle |=|\sum_{i=1}^{k} \lambda_{i} \langle g, \Df_{i} \rangle|
\leq \sum_{i=1}^{k}|\lambda_{i}| \sup \{ |\langle g, \Df \rangle|: f
\in \F \}$. Thus $|\langle g,f \rangle |\leq \|f\|_{0}\|g\|_{BAC}$
for every $f$, so that $\|g\|_{0}^{*} \leq \|g\|_{BAC}$.

For the other direction, suppose $\|g\|_{BAC}=1$. Then for every
$\epsilon>0$, there exists $f \in \F$ such that $|\langle g, \Df
\rangle| \geq 1-\epsilon$. Note that $ \|\Df \|_{0}\geq 1$. Thus
$\|g \|_{0}^{*} \geq 1-\epsilon$, for any $\epsilon>0$. This shows
that $\|g\|_{0}^{*} \geq \|g\|_{BAC}$ for any $g \in \R^{G}$.
Therefore, $\|\cdot\|_{BAC}=\|\cdot\|_{0}^{*}$.
\end{proof}
We now see that the name ``quasi-algebra predual'' is justified by the following:
\begin{lemma}\label{BAC}
If $\psi \in \R^{G}$ is such that $\| \psi \|_{BAC}^{*} \leq 1$, then $\|\psi^{m} \|^{*} \leq C(m)$.
\end{lemma}
\begin{proof}
If $\| \psi \|_{BAC}^{*} \leq 1$, then for every $\epsilon>0$, $\psi$ can be written as a linear combination of functions $\Df, f \in \F$ and the absolute value of the coefficients adding up to less than $1+\epsilon$. Thus $\psi^{m}$ can be written as a linear combinations of products of $m$ functions from $\{\Df:f \in \F\}$, with the absolute value of the coefficients adding up to less than $(1+\epsilon)^{m}$. Since the $\|\cdot\|^{*}$ norm of every product of $m$ functions from $\{\Df:f \in \F\}$ is at most $C(m)$, we conclude that $\|\psi^{m} \|^{*} \leq C(m)$.
\end{proof}

Specializing to the case where $\F$ is the set of all nonnegative functions bounded by a function $\nu \in \R^{G}$, we claim that any function from $\F$ can be written as the sum of a bounded function and another function small under $\|\cdot \|$. This is the content of the Green-Tao-Ziegler structure theorem.

\begin{theorem}[Green-Tao-Ziegler structure theorem, \cite{gt-primes}, \cite{tao-ziegler}]\label{gtz}
For every $\eta >0$, there is $\epsilon=\epsilon(\eta, C,c)>0$ such that the following holds:
Let $\nu$ be a measure on $G$ such that $\|\nu-1\|<\epsilon$, $\E_{G}(\nu)\leq 1+\eta$, and all properties in Definition \ref{QAP} hold for $\F=\{f: 0 \leq f \leq \nu \}$. Then for every function $f \in \F$, $f$ can be decomposed as $f=g+h$, where $0 \leq g \leq 1+\eta$ and $\|h\| \leq \eta$.
\end{theorem}

\begin{proof}
Suppose such a decomposition doesn't exist. Since $\|h\|_{BAC} \leq c(\eta)$ implies $\|h\| \leq \eta$, this implies that $f$ cannot be expressed as the sum of elements from two convex sets $X_1= \{ 0\leq g \leq 1+\eta\}$ and $X_2= \{\|h\|_{BAC} \leq c(\eta) \}$ in $\mathbf{R}^{G}$.
\begin{claim}
There is a function $\psi \in \R^{G}$ such that $\langle f, \psi \rangle >1$, but $\langle g, \psi \rangle \leq 1$ and $\langle 1, \psi \rangle \leq 1$ for every $g \in X_1, h \in X_2$.
\end{claim}
\begin{proof}
Let $X=X_1+X_2$, then $X$ is convex and closed. We invoke the following form of the Hahn-Banach theorem: if $f \notin X$, then there is a linear functional $\langle \cdot, \psi \rangle$ on $\R^{G}$ such that $\langle f, \psi \rangle >1$ and $\langle g, \psi \rangle \leq 1$ for every $g \in X$. Since $X_1$ and $X_2$ both contain 0, $X_1$ and $X_2$ are contained in $X$ and the claim follows.
\end{proof}
The condition $\langle g, \psi \rangle \leq 1$ for every $g \in X_{1}$ implies that $\E_{G} \psi_{+} \leq \frac{1}{1+\eta}$, where $\psi_{+}(x)=\max(0, \psi(x))$.
The condition $\langle h, \psi \rangle \leq 1$ for every $h \in X_{2}$ implies that $\|\psi\|_{BAC}^{*}\leq c(\eta)^{-1}$.
\begin{claim}
For any $\eta'>0$, there is a polynomial $P=P(\eta,\eta',C,c)$ and a constant $R=R(\eta,\eta',C,c)$ such that $\|P\psi-\psi_{+}\|_{\infty} \leq \eta'$ and $\|P\psi\|^{*} \leq R$.
\end{claim}
\begin{proof}
Since $\|\psi\|_{BAC}^{*}\leq c(\eta)^{-1}$, by Lemma \ref{BAC} we have $\|\psi\|^{*}\leq C_1=C(1)c(\eta)^{-1}$.
By Weierstrass' approximation theorem, there is a polynomial $P(x)=a_{n}x^{n}+\cdots+a_0$ such that $|P(x)- \max(0,x)| \leq \eta'$ for every $x\in [-C_1,C_1]$. Then clearly $\|P\psi-\psi_{+} \|_{\infty} \leq \eta'$. Next we claim that $\|P\psi\|^{*}$ is bounded (independent of $\psi$). By the triangle inequality it suffices to show this for $\| \psi^{m} \|^{*}$ for each $m$. But this follows from Lemma \ref{BAC}.
\end{proof}

We now have $1 < \E_{G} f\psi_{+} \leq \E_{G} \nu \psi_{+}$. We split the later as  $$\E_{G} \nu \psi_{+}= \E_{G}\psi_{+} + \E_{G}(\nu-1)P\psi + \E_{G}(\nu-1)(\psi_{+}-P\psi)$$
Also, $|\E_{G} (\nu-1) P\psi| \leq \| \nu-1 \| \|P\psi\|^{*} \leq \epsilon R$, and $|\E_{G} \nu (\psi_{+}-P\psi)| \leq (\E_{G}\nu) \|P\psi-\psi_{+} \|_{\infty} \leq \eta' (1+\eta)$.
Thus $1 \leq \frac{1}{\eta+1} + \eta' (1+\eta) + \epsilon R$ . If we fix a small value of $\eta'$ (e.g. $\eta'=\eta/12$ will do), then this is a contradiction is $\epsilon$ is small enough.
\end{proof}
Let us now formulate the result in the setting of Gowers norm and pseudorandom measures on $\FqN$:
\begin{corollary}
Let $\nu$ be a pseudorandom measure on $\FqN$. Then for every $\eta>0$, for $N$ sufficiently large, every function $\phi$ on $\FqN$ such that $0 \leq \phi \leq \nu$, can be decomposed as $\phi=\phi_1+\phi_2$, where $0 \leq \phi_1 \leq 2+\eta$ and $\phi_2$ is uniform in the sense that $\| \phi_2 \|_{U^{K-1}} \leq \eta$.
\end{corollary}
\begin{proof}
If $0 \leq \phi \leq \nu$, then $0 \leq \frac{\phi}{2} \leq \nu_{1/2}=\frac{\nu+1}{2}$. We already know that $\nu_{1/2}$ is also pseudorandom. Let $G=\FqN$ and $\F$ be the space of all nonnegative functions bounded by $\nu_{1/2}$. Let us check that the normalized Gowers $U^{K-1}$ norm $\|\phi\|=\frac{1}{2} \|\phi\|_{U^{K-1}}$ is quasi-algebra predual with respect to $\F$, where $\Dphi=\frac{1}{2} \mathcal{D}_{K-1}\phi$. Thanks to Lemmas \ref{property1}, \ref{property2}, \ref{property3}, the first three conditions in Definition \ref{QAP} are met. The only thing left to check is the forth condition, i.e. the set of dual functions $\mathcal{D}\phi$ spans $\R^{G}$. Note that if $\phi$ is a point mass, then $\Dphi$ is also a point mass (at the same point). Since $\nu_{1/2}$ is pointwise positive\footnote{This is the sole reason why we work with $\nu_{1/2}$ rather than with $\nu$.}, $\{\mathcal{D}\phi: \phi \in \F \}$ contains masses at every point of $G$, hence spans $\R^{G}$.

By Theorem \ref{gtz}, there is $\epsilon = \epsilon(\eta)>0$ such that we have a decomposition $$\frac{\phi}{2}=\phi_1+\phi_2 \textrm{ where $0 \leq \phi_1 \leq 1+\frac{\eta}{2}$ and $\|\phi_2\|_{U^{K-1}} \leq \frac{\eta}{2} $} $$  as soon as $\E_{G}\nu_1 \leq 1+\eta$ and $\|\nu_1-1\|\leq \epsilon$. But this is always true since $\nu_1$ is a pseudorandom measure. Such a decomposition for $\phi/2$ gives the desired composition for $\phi$.
\end{proof}
\begin{remark}
If instead of $\nu_{1/2}$ we consider $\nu_{\alpha}=(1-\alpha)\nu+\alpha$, where $\alpha>0$ is sufficiently small depending on $\eta$, we can actually show that there is a decomposition $\phi=\phi_1+\phi_2$, where $0 \leq \phi_1 \leq 1+\eta$ and $\| \phi_2  \|_{U^{K-1}} \leq \eta$, but this is not important.
\end{remark}
With this in hand, we can now prove Theorem \ref{FFESzemeredi}:
\begin{proof}[Proof of Theorem \ref{FFESzemeredi} using the Green-Tao-Ziegler structure theorem.]
We know that for every $\eta>0$, for $N$ sufficiently large (depending on $\eta$), every function $\phi$ bounded by a pseudorandom measure on $\FqN$ can be decomposed as $\phi=\phi_{1}+\phi_{2}$, where $0\leq\phi_1\leq 2+\eta$ and $\|\phi_2\|_{U^{K-1}} \leq \eta$. In particular $|\E \phi_2| \leq \eta$, so that if $\E \phi \geq \delta$, then $\E \phi_1 \geq \delta - \eta$.
Write
$$\E \left(\prod_{P \in \Gk} \phi(f+Pg)| f, g \in \FqN \right)=\E \left(\prod_{P \in \Gk} \phi_1(f+Pg)| f, g \in \FqN \right) +\textrm{ $(2^{K}-1)$ other terms}$$
The other terms are of the form $\E \left(\prod_{P \in \Gk} \phi_{P}(f+Pg)| f, g \in \FqN \right)$ where each $\phi_{P}=\phi_1$ or $\phi_2$, and not all $\phi_{P}$ are equal to $\phi_1$.

Since $\phi_1$ is bounded pointwise by $2+\eta$ and $\phi_2$ is bounded pointwise by $\max(\nu, 2+\eta) \leq 3+\nu$ in absolute value, by Proposition \ref{gvN2}, these terms are at most $4^{K}\| \phi_2\|_{U^{K-1}}+o(1)$ in absolute value.

On the other hand, by Theorem \ref{FFSzemeredi}, $\E \left(\prod_{P \in \Gk} \phi_1(f+Pg)| f, g \in \FqN \right) \geq (2+\eta)^{K} c(\frac{\delta-\eta}{2+\eta})$. Hence
$$\E \left(\prod_{P \in \Gk} \phi(f+Pg)| f, g \in \FqN \right) \geq (2+\eta)^{K} c\left(\frac{\delta-\eta}{2+\eta}\right) - (2^{K}-1)4^{K}\eta - o(1)$$
By choosing $\eta$ appropriately small, the main term on the right hand side is positive, so that there is a positive constant $c'(\delta)$ such that $\E \left(\prod_{P \in \Gk} \phi(f+Pg)| f, g \in \FqN \right) \geq c'(\delta) - o(1)$ for every function $\phi$ on $\FqN$ bounded by a pseudorandom measure.
\end{proof}
\begin{remark}By running the argument carefully (e.g. by modifying $\phi_1$ so that it is bounded above by exactly 1, and its average is exactly $\delta$) we can show that actually $c'(\delta)$ can be taken to be $c(\delta)$. However, there is little point in doing so since we don't have an explicit value for $c(\delta)$.
\end{remark}

\section{Elementary arithmetic in $\Fq[t]$} \label{S8}
In this section we will describe some basic arithmetic properties of
$\Fq[t]$, introduce arithmetic functions on $\Fq[t]$ and prove some
preliminary lemmas relevant to the construction of a pseudorandom
measure in Section \ref{S9}. We assume from now on that
polynomials denoted by the letter $P$ (such as $P,P'$, or $P_{i}$)
will stand for monic, irreducible polynomials.

The units of the ring $\Fq[t]$ is $\Fq \setminus \{0\}$. Similarly to the
integers, $\Fq[t]$ is a unique factorization domain. More precisely,
every $f \in \Fq[t]$ can be written uniquely as $f=cP_1^{\alpha_1}
\cdots \P_{m}^{\alpha_{m}}$, where $c \in \Fq, \alpha_{i} \in \Z^{+}$
and the $P_{i}$ are monic, irreducible polynomials. We can now
introduce arithmetic functions on $\Fq[t]$:
\begin{itemize}
    \item The Euler totient function $\Phi(f)$, is the number of polynomials
of degree less than $\d(f)$ which are relatively prime to $f$. Then
we have the following formula for $\Phi(f)$ in terms of its prime
factorization: $\Phi(f)=|f|\prod_{P|f}\left(1-\frac{1}{|P|} \right)=
\prod_{i=1}^{m}\frac{|P_{i}^{\alpha_{i}+1}|-1}{|P_{i}|-1}$
    \item The Mobius function $\mu(f)=\left\{
                               \begin{array}{ll}
                                 (-1)^{m}, & \hbox{if $\alpha_{i}=1$ for every $i=1, \ldots, m$;} \\
                                 0, & \hbox{otherwise.}
                               \end{array}
                             \right.$
  \item The von Mangoldt function $\Lambda(f)=\left\{
                               \begin{array}{ll}
                                 \d(f), & \hbox{if $m=1$;} \\
                                 0, & \hbox{otherwise.}
                               \end{array}
                             \right.$
  \item $d(f)$, the number of monic divisors of $f$. We have the following formula: $d(f)=\prod_{i=1}^{m}(\alpha_{i}+1)$.
  \item For $d_{1}, \ldots, d_{m} \in \Fq[t], d_{i} \neq 0$, denote
by $[d_{1}, \ldots, d_{m}]$ the least common divisor of $d_{1},
\ldots, d_{m}$, in other words, the polynomial of smallest
degree that is divisible by $d_{i}$ for every $i=1,\ldots,m$ (which is defined up to multiplication by an element of $\Fq \setminus \{0\}$).
\end{itemize}

The zeta function $\zeta_{q}$ of $\Fq[t]$ is defined by
$\zeta_{q}(s)=\sum_{f \textrm{ monic}}\frac{1}{|f|^{s}}$ for any $s
\in \C$ such that $\Re s>1$. We have the following closed form for
the zeta function: $\zeta_{q}(s)=\frac{1}{1-q^{1-s}}$ for $\Re s>1$.
Thus it can be analytically continued on the whole plane, with a
simple pole at $s=1$, at which the residue is $\frac{1}{\log q}$.

Similarly to the Riemann zeta function, $\zeta_{q}$ admits a
factorization as an Euler product: $\zeta_{q}(s)=\prod_{P}\left(1-
\frac{1}{|P|} \right)^{-1}$.

We have the following analog of the prime number theorem
\cite{rosen}

\begin{proposition}[Prime number theorem for function
fields]\label{PNT}
Let $\pi_{q}(N)$ be the number of irreducible polynomials of degree $N$
in $\Fq[t]$. Then $\pi_{q}(N)=(q-1)\frac{q^{N}}{N}+O\left(
\frac{q^{N/2}}{N}\right)$.
\end{proposition}
\begin{corollary}\label{PNT2}
$\sum_{\d(P)\leq N}\frac{1}{|P|}=\log N+O_{q}(1)$
\end{corollary}
\begin{proof}
We have
\begin{eqnarray}
\sum_{\d(P)\leq N} \frac{1}{|P|} &=& 1+\sum_{n=1}^{N}\frac{\pi_{q}(n)}{q^{n}} \nonumber \\
 &=& 1+\sum_{n=1}^{N}\left(\frac{1}{n}+O\left(\frac{q^{-n/2}}{n}\right)\right) \nonumber \\
 &=& \log N +O(1) \nonumber
\end{eqnarray}
\end{proof}
More generally, we have the following analog of Dirichlet's theorem
on primes in arithmetic progressions (with a much better error term
than its integer counterpart, thanks to the Riemann hypothesis for curves
over a finite field):

\begin{proposition}[Dirichlet's theorem for function fields]\label{Dirichlet}
Let $a,r \in \Fq[t]$ be relatively prime, $\deg(m)>0$. Let
$\pi_{q}(N;a,r)$ be the number of irreducible polynomials of degree $N$
in $\Fq[t]$ which are congruent to $r$ (modulo $a$). Then
$\pi_{q}(N;a,r)=(q-1)\frac{1}{\Phi(m)} \frac{q^{N}}{N}+O\left(
\frac{q^{n/2}}{n}\right)$.
\end{proposition}

We will need the following two lemmas in our construction of the
function $\tau$ in the correlation condition.

\begin{lemma}[Divisor bound] \label{divisor}
Let $f \in \Fq[t]$. Suppose $\d(f)=N$. Then $d(f)$, the number of
divisors of $f$, satisfies $d(f) \leq q^{O_{q}\left( \frac{N}{\log
N}\right)}$.
\end{lemma}
\begin{proof}If $f$ has the factorization
$f=c\prod_{i=1}^{m}P_{i}^{\alpha_{i}}$, then
$d(f)=\prod_{i=1}^{m}(\alpha_{i}+1)$. Therefore,
$\frac{d(f)}{|f|^{\epsilon}}=\prod_{i=1}^{m}
\frac{\alpha_{i}+1}{|P_{i}|^{\epsilon \alpha}}$, where $\epsilon$ is
to be chosen later, possibly depending on $f$.

Note that, if $\d(P_{i}) \geq 1/\epsilon$, then
$$\frac{\alpha_{i}+1}{|P_{i}|^{\epsilon \alpha_{i}}} \leq
\frac{\alpha_{i}+1}{q^{\alpha_{i}}} \leq 1$$ If $\d(P_{i}) <
1/\epsilon$, then $$\frac{\alpha_{i}+1}{|P_{i}|^{\epsilon
\alpha_{i}}} \leq \frac{\alpha_{i}+1}{q^{\epsilon
\alpha_{i}}} \leq \frac{q^{2\sqrt{\alpha_{i}}}}{q^{\epsilon
\alpha_{i}}} \leq q^{1/\epsilon}$$
Since the second case can occur for at most $q^{1/\epsilon}$ values of $P_{i}$, we have
$\frac{d(f)}{q^{N
\epsilon}} \leq (q^{1/\epsilon})^{q^{1/\epsilon}}=q^{N \epsilon +
\frac{1}{\epsilon}q^{1/\epsilon}}$.
Thus for every $\epsilon>0$,
$$d(f) \leq q^{N \epsilon + \frac{1}{\epsilon}q^{1/\epsilon}}$$
for every $\epsilon >0$. If we choose $\epsilon=1/\log N$, then we have $d(f)\leq q^{O_{q}\left(
\frac{N}{\log N}\right)}$, as required.
\end{proof}
\begin{lemma}\label{tz}Let $\mathcal{S}$ be a finite set of
irreducible polynomials in $\Fq[t]$, then for every $K$,
$$\exp \left( \sum_{P \in \mathcal{S}} \frac{1}{|P|} \right)=O_{K} \left( \sum_{P \in \mathcal{S}} \frac{\log^{K}|P|}{|P|} \right) $$
\end{lemma}
This bound is perhaps surprising, since it is uniform over all
finite subset of the irreducible polynomials.
\begin{proof}
We have
\begin{eqnarray}
 \exp \left( K \sum_{P \in \mathcal{S}} \frac{1}{|P|} \right)&=& 1+\sum_{n=1}^{\infty} \frac{K^{n}}{n!} \sum_{P_1, \ldots, P_{n}\in \mathcal{S}} \frac{1}{|P_1 \cdots P_{n}|} \nonumber \\
&\leq& 1+\sum_{n=1}^{\infty} \frac{K^{n}}{(n-1)!} \sum_{P \in
\mathcal{S}} \sum_{ \substack{P_1, \ldots, P_{n-1} \in
\mathcal{S},\\
\d(P_{i}) \leq \d(P)}}
\frac{1}{|PP_1 \cdots P_{n-1}|} \nonumber \\
&=& 1+ \sum_{P\in \mathcal{S}} \frac{1}{|P|} \sum_{n=1}^{\infty}
\frac{K^{n}}{(n-1)!} \left( \sum_{\substack{P' \in \mathcal{S},\\
\d(P') \leq \d(P)}}
\frac{1}{|P'|} \right)^{n-1} \nonumber
\end{eqnarray}
By Corollary \ref{PNT2}, we have that $\sum_{\substack{P' \in \mathcal{S},\\
\d(P') \leq \d(P)}} \frac{1}{|P'|} \ll \log \d(P)$. Hence
\begin{eqnarray}
\exp \left( \sum_{P \in \mathcal{S}} \frac{1}{|P|} \right)
&\ll& \sum_{P\in \mathcal{S}} \frac{1}{|P|} \sum_{n=1}^{\infty}
\frac{K^{n}}{(n-1)!} \log^{n-1}\d(P) \nonumber \\
&\ll_{K}& \sum_{P \in \mathcal{S}} \frac{1}{|P|} \exp(K \log(\d(P))) \nonumber \\
&\ll_{K}& \sum_{P \in \mathcal{S}} \frac{\log^{K}|P|}{|P|} \nonumber
\end{eqnarray}
as required.
\end{proof}

In our proof of the Goldston-Y{\i}ld{\i}r{\i}m estimates (Propositions
\ref{gy}, \ref{gy2}, \ref{gy3}) in the next sections, we will be concerned with Euler
products in several variables, i.e of the form $\prod_{P}\left( 1-\sum_{j=1}^{n}
\frac{c_{P,j}}{|P|^{1+s_{j}}} \right)$, as $\Re s_{j}>0$ and
$s_{j}\rightarrow 0$ uniformly. The following lemma gives an
asymptotic formula for such Euler products.

\begin{lemma}\label{euler}
Let $P$ range over monic irreducible polynomials in $\Fq[t]$. For every $P$ let $c_{P,1}, \ldots, c_{P,n}$ be real numbers such that $|c_{P,j}| \leq 1$ and $c_{P,j}=c_{j}$ for $P$ outside a finite set $\mathcal{S}$. Let $s_{1}, \ldots, s_{n} \in \C$ be such that $\Re s_{j}>0$ and $s_{j}=o(1)$ uniformly. Then
$$ \prod_{P}\left( 1-\sum_{j=1}^{n} \frac{c_{P,j}}{|P|^{1+s_{j}}} \right)=G(1+o_{n}(1))\prod_{P \in \mathcal{S}}\left( 1+O_{n}\left( \frac{1}{|P|}\right) \right) \prod_{j=1}^{n}\left( s_{j} \log q \right)^{c_{j}}$$
where $G=\prod_{P}
\left(1-\frac{\sum_{j=1}^{n}c_{P,j}}{|P|}\right)\left(1-\frac{1}{|P|}
\right)^{-(c_1+\cdots+c_{n})}$
\end{lemma}
Note that the $O_{n}$ and $o(1)$ depends only on $n$ and the rate $s_{1}, \ldots, s_{n} \rightarrow 0$ and not on the exceptional set $\mathcal{S}$.
\begin{proof}
Note that $\left(1-\frac{\sum_{j=1}^{n}c_{P,j}}{|P|}\right)\left(1-\frac{1}{|P|}
\right)^{-(c_1+\cdots+c_{n})}=1+ O_{n}\left( \frac{1}{|P|}\right)$ if $P \in \mathcal{S}$ and $1+ O_{n}\left( \frac{1}{|P|^2}\right)$ if $P \not \in \mathcal{S}$. In particular the product defining $G$ converges.

Let us now look at the expression
$$\left(1-\sum_{j=1}^{n} \frac{c_{P,j}}{|P|^{1+s_{j}}} \right) \prod_{j=1}^{n}\left(1-\frac{1}{|P|^{1+s_{j}}} \right)^{-c_{j}}$$
For $P$ outside of $\mathcal{S}$, an easy calculation (by calculating the partial derivative of the expression with respect to each $s_{j}$) shows that it is equal to $$\left(1-\frac{\sum_{j=1}^{n}c_{P,j}}{|P|}\right)\left(1-\frac{1}{|P|}
\right)^{-(c_1+\cdots+c_{n})}\left( 1+o_{n}\left( \frac{\log |P|}{|P|^2}\right) \right)$$
For $P \in \mathcal{S}$, we just bound it crudely by $1+O_{n}\left( \frac{1}{|P|}\right)$, which is equal to
$$\left(1-\frac{\sum_{j=1}^{n}c_{P,j}}{|P|}\right)\left(1-\frac{1}{|P|}
\right)^{-(c_1+\cdots+c_{n})}\left( 1+O_{n}\left( \frac{1}{|P|}\right) \right)$$
Multiplying these estimates over all $P$, (and noting that $\prod_{P} \left( 1+o_{n}\left( \frac{\log |P|}{|P|^2}\right) \right) =1+o_{n}(1)$), we have
$$ \prod_{P}\left( 1-\sum_{j=1}^{n} \frac{c_{P,j}}{|P|^{1+s_{j}}} \right) = G(1+o_{n}(1)) \prod_{j=1}^{n}\zeta_{q}(1+s_{j})^{-c_{j}} \prod_{P \in \mathcal{S}}\left(1+O_{n}\left( \frac{1}{|P|}\right) \right)$$
Writing out $\zeta_{q}(1+s_{j})^{-1}=1-q^{-s_{j}}=(1+o(1))s_{j} \log q$, we have the desired estimate.
\end{proof}

\section{A pseudorandom measure that majorizes the irreducible polynomials} \label{S9}
In this section we prove Theorem \ref{FFGY} by constructing a pseudorandom measure $\nu$. The proof of its pseudorandomness is however deferred to the next two sections. Recall that our task is to find a pseudorandom
measure $\nu$ such that $\nu$ majorizes a function $\phi$ which is
supported on $\A$ and such that $\E(\phi|\FqN) \geq \delta$, where
$\overline{d}_{\P}(\A)>0$, and $\delta$ is a positive constant
depending on $\overline{d}_{\P}(\A)$ alone. Throughout this whole
section and the next two, polynomials denoted by the letter $d$
(such as $d, d'$ or $d_{i}$) will stand for monic polynomials.

Let's fix once and for all
\begin{itemize}
    \item $R=\alpha N$, where $\alpha$ is a small constant depending only on $k$.
    \item $w=w(N)$, a function tending sufficiently slowly to infinity. We may take
$w(N) \ll \log N$.
    \item $W=\prod_{\d(P)<\omega}P$. We have that $W(t)=t^{q^{w}}-t$, so that\footnote{The introduction of $W$,
alluded to earlier as the $W$-trick, is meant to absorb small
irreducible polynomials arising in the the linear
forms condition. Except for this technical reason, for the most part we can go through the arguments pretending that $W=1$ without losing
the general idea.} $\d(W)\ll N$. We will see that eventually we can take $w$ to be a sufficiently large number, hence $W$ to be a sufficiently large polynomial.
    \item $\chi: \mathbf{R} \rightarrow \mathbf{R}$ a smooth function\footnote{Goldston-Y{\i}ld{\i}r{\i}m used a truncated sum corresponding
to $\chi(x)=\max(1-|x|, 0)$. As observed by Tao
\cite{tao-gy}, the use of a smooth function allows us to perform
Fourier analysis.} supported on [-1,1] such that $\chi(0)>0$ and
$\int_{0}^{\infty}(\chi'(x))^2 dx=1$.

    \item $\Lambda_{R}(f)= \sum_{\substack{d | f, \\ \d(d)<R}} \mu(d) \chi\left( \frac{\d(d)}{R}
\right)$, the Goldston-Y{\i}ld{\i}r{\i}m divisor sum.
    \item $\nu(f)= \nu_{b}(f)=R \frac{\Phi(W)}{|W|}
\Lambda_R(Wf+b)^2$ for some appropriate $b$ such that $0<\d(b)<\d(W), \gcd(b,W)=1$ to be chosen later.
\end{itemize}
\begin{proof}[Proof of Theorem \ref{FFGY} under the assumption that $\nu$ is pseudorandom]
Notice that if $f$ is irreducible and $\d(f) \geq R$ then
$\Lambda_{R}(f)=\chi(0)$. For $f \in \GN$ let
$$\phi(f)=\phi_{b}(f)= \left\{
     \begin{array}{ll}
       \chi(0)^2 \frac{\Phi(W)}{|W|} R, & \hbox{if $Wf+b$ is irreducible and $\d(Wf+b) \geq R$;} \\
       0, & \hbox{otherwise.}
     \end{array}
   \right.$$
Then clearly $0 \leq \phi \leq \nu$ and $\|\phi\|_{\infty} \ll N$.
Thus it suffices to find $b$ such that $\sum_{f \in
\GN}\phi_{b}(f)1_{Wf+b \in \A}\geq \delta q^{N}$ for some constant
$\delta>0$.

Let us take the sum $\sum_{b} \sum_{f \in \GN}\phi_{b}(f)1_{Wf+b \in
\A}$ over all $b$ such that $\d(b)<\d(W), \gcd(b,W)=1$. It is easy
to see that it is equal to $$\sharp \{h \in \A,
R \leq \d(h) \leq N+\d(W) \} \frac{\Phi(W)}{|W|}R$$ By the prime number theorem in
$\Fq[t]$ (Theorem \ref{PNT}), $\sharp \{h \in \P, R \leq \d(h) \leq
N+\d(W) \}=(q-1)\frac{q^{N}|W|}{N+\d(W)}(1+o(1))$

Since $N+\d(W)$ increases at most linearly in $N$, and since
$\overline{d}_{\P}(\A)>0$, we conclude that there is a constant
$\delta>0$ depending only on $\A$ such that $\sum_{b} \sum_{f \in
\GN}\phi_{b}(f)1_{Wf+b \in \A} \geq \frac{\Phi(W)}{|W|}\delta q^{N}$
infinitely often\footnote{This is always true if the limit $d_{\P}(\A)=\lim_{N \rightarrow \infty} \frac{ \# \{ f \in \A, \d(f) < N \} }{\# \{ f \in \P, \d(f) < N \}}$ exists. If not, then this can be false if we allow $W$ to tend to infinity. However, as already mentioned earlier, we can eventually take $W$ to be a constant, so that the argument remains valid.}. Thus, for infinitely many $N$, we can find $b$
such that $\sum_{f \in \GN}\phi_{b}(f)1_{Wf+b \in \A}\geq \delta
q^{N}$, as required.
\end{proof} From now on let us assume without loss of
generality that $b=1$. (Note that if $\A=\mathcal{P}$ we can always
take $b=1$, thanks to Dirichlet's theorem in $\Fq[t]$ (Theorem
\ref{Dirichlet})).

The only thing missing from the conclusions of Theorem \ref{FFGY} is
to check that $\nu$ is indeed a pseudorandom measure,
i.e. it satisfies the linear forms condition and the correlation
condition. This will be done in the next two sections. In order to
do so, we will need estimates on sums of the form $\sum
\Lambda_{R}(\psi_1)\ldots \Lambda_{R}(\psi_{n})$ where the $\psi_{i}$
are linear forms. The following proposition shows us how to deal
with sums of this kind.

\begin{proposition}[Goldston-Y{\i}ld{\i}r{\i}m estimates]\label{gy}
Given $J_1, \ldots, J_{n} \in \Fq[t]$ not necessarily distinct. Let
$r$ be the number of distinct elements in $\{J_1, \ldots, J_{n}\}$.
Also, for every monic, irreducible $P\in \Fq[t]$, let $\alpha_{P}$
be the number of distinct residue classes modulo $P$ occupied by
$J_1, \ldots, J_{n}$. Put $\Delta=\Delta(J_1,\ldots,J_{n})=\prod
(J_{i}-J_{i'})$, where the product is taken over all couples
$(J_{i},J_{i'})$ such that $J_{i} \neq J_{i'}$. Then as
$N\rightarrow \infty$,
\begin{equation}\label{gye}
\sum_{f \in \GN}
\Lambda_{R}(f+J_1)\cdots\Lambda_{R}(f+J_{n})=CGH(1+o_{n}(1))q^{N}\left( \frac{\log
q}{R} \right)^{r}
\end{equation}
where $C$ is a computable constant (not depending on $N, J_1,
\ldots, J_{n}$ but only on the multiplicities of the $J_{i}$ and
$\chi$), $G$ is the ``arithmetic factor'' $\prod_{P} \left(
1-\frac{1}{|P|} \right)^{-r}\left(1-\frac{\alpha_{P}}{|P|} \right)$, and $H=\prod_{P|\Delta}\left(1+O_{n}\left(\frac{1}{|P|} \right) \right)$.
\end{proposition}
\begin{remarks}This Proposition illustrates how the Goldston-Y{\i}ld{\i}r{\i}m
method works. We will not apply this Proposition directly, (since we
will be incorporating the $W$-trick), but rather its variants
(Propositions \ref{gy2} and \ref{gy3}), for which only minor
modifications are needed.
\end{remarks}
\begin{proof}
Writing out the definition of $\Lambda_{R}$, we see that the left
hand side of (\ref{gye}) is
\begin{equation} \label{eq1}
\sum_{d_1,\ldots,d_{n} \in \GN} \left( \prod_{i=1}^{n} \mu (d_{i})
\chi \left( \frac{\d(d_{i})}{R} \right) \right) \sum_{f \in
\GN}1_{d_{i}|f+J_{i} \forall i=1, \ldots, n } \end{equation} Note
that since $\chi$ is supported on $[-1,1]$, the summation over
$d_1,\ldots,d_{n} \in \GN$ is the same as the summation over
$d_1,\ldots,d_{n} \in \mathbf{G}_{R}$. Also, because of the appearance of the function $\mu$, only squarefree $d_{i}$ are involved.

Suppose $R$ is sufficiently small compared to $N$, say $nR<N$. Then for every $d_1,\ldots,d_{n}
\in \mathbf{G}_{R}$, we have $|[d_1,\ldots, d_{n}]|<q^{N}$.
Therefore
\begin{equation} \label{cc}
\sum_{f \in \GN}1_{d_{i}|f+J_{i} \forall i=1, \ldots, n }=\frac{g(d_1,\ldots,d_{n})}{|[d_1,\ldots,d_{n}]|}q^{N}
\end{equation}
where $g(d_1,\ldots,d_{n})$ is the number of solutions in
$\mathbf{G}_{\d([d_1,\ldots,d_{n}])}$ of the system of congruences
$f+J_{i} \equiv 0 \pmod{d_{i}}$ for every $i=1,\ldots,k$. We will
come back to the analysis of $g(d_1,\ldots,d_{n})$ later.

We can now rewrite the expression (\ref{eq1}) as
\begin{equation} \label{eq2}
q^{N} \sum_{d_1,\ldots,d_{n}}\frac{g(d_1,\ldots,d_{n})}{|[d_1,\ldots,d_{n}]|}
\prod_{i=1}^{k}\mu(d_{i}) \chi\left( \frac{\d(d_{i})}{R}\right)
\end{equation}

Write\footnote{The appearance of $q$ here is for mere aesthetic
reasons.} $\chi(x)=\int_{-\infty}^{\infty}q^{-(1+it)x}\psi(t)dt$ (in
other words, $\psi(t)= \log{q} \widehat{q^{x} \chi(x)}(t \log q)$,
where $\psi$ is rapidly decreasing, i.e.
$\psi(t)=O_{A}((1+|t|)^{-A})$ for every $A$. Thus for every $m=1,
\ldots, n$,
\begin{eqnarray}
\chi\left( \frac{\d(d_{m})}{R} \right) &=& \int_{-\infty}^{\infty}q^{-\frac{1+it_{m}}{R}\d(d_{m})}\psi(t_{m}) dt_{m} \nonumber \\
&=& \int_{-\infty}^{\infty}|d_{m}|^{-\frac{1+it_{m}}{R}}\psi(t_{m}) dt_{m} \nonumber \\
&=& \int_{-\sqrt{R}}^{\sqrt{R}}|d_{m}|^{-\frac{1+it_{m}}{R}}\psi(t_{m}) dt_{m} + O_{A}(|d_{m}|^{-1/R}R^{-A})\nonumber
\end{eqnarray}
We split the expression (\ref{eq2}) into a main term of
$$q^{N} \sum_{d_1,\ldots,d_{n}}\frac{g(d_1,\ldots,d_{n})}{|[d_1,\ldots,d_{n}]|} \prod_{m=1}^{n}\mu(d_{m}) \int_{-\sqrt{R}}^{\sqrt{R}}|d_{m}|^{-\frac{1+it_{m}}{R}}\psi(t_{m}) dt_{m}$$
plus an error term, which is
\begin{eqnarray}
& &q^{N} \sum_{d_1,\ldots,d_{n}}\frac{g(d_1,\ldots,d_{n})}{|[d_1,\ldots,d_{n}]|} O_{A}(R^{-A}|d_1\cdots d_{n}|^{-1/R}) \nonumber\\
&\ll_{A}& q^{N} R^{-A}\sum_{d_1,\ldots,d_{n}}
\frac{|d_1\cdots d_{n}|^{-1/R}}{|[d_1,\ldots,d_{n}]|} \label{eq3}
\end{eqnarray}
Note that $\sum_{d_1,\ldots,d_{n}}\frac{|d_1\cdots
d_{n}|^{-1/R}}{|[d_1,\ldots,d_{n}]|}$ factors as
\begin{eqnarray}
& & \prod_{P} \left(\sum_{d_1,\ldots,d_{n} \in \{1, P, P^2,\ldots
\}}
\frac{|d_1\cdots d_{n}|^{-1/R}}{|[d_1,\ldots,d_{n}]|} \right) \nonumber \\
&=& \prod_{P} \left( 1+ \frac{n}{|P|^{1+1/R}} +
O_{n}\left(\frac{1}{|P|^2}\right)
\right) \nonumber \\
&=& \prod_{P} \left( 1- \frac{1}{|P|^{1+1/R}} \right)^{-n} \left(1+
O_{n}\left(\frac{1}{|P|^2} \right) \right) \nonumber
\end{eqnarray}

Note that the product $\left(1+O\left( \frac{1}{|P|^2}\right)
\right)$ is absolutely convergent, while $\prod_{P}
\left(1-\frac{1}{|P|^{1+\frac{1}{R}}} \right) =
\zeta_{q}(1+\frac{1}{R})= \frac{1}{1-q^{-1/R}}=O(R)$. Thus by
choosing $A \geq n$, we see that the error term (\ref{eq3}) is
$o(q^{N})$. In particular the sum in (\ref{eq2}) converges absolutely.

Therefore, it suffices to show that
$$\sum_{d_1,\ldots,d_{n}}\frac{g(d_1,\ldots,d_{n})}{[d_1,\ldots,d_{n}]} \prod_{m=1}^{k}\mu(d_{m}) \prod_{m=1}^{k} \int_{-\sqrt{R}}^{\sqrt{R}}|d_{m}|^{-\frac{1+it_{m}}{R}}\psi(t_{m}) dt_{m}=CGH(1+o_{n}(1))\left( \frac{\log q}{R} \right)^{r}$$
Note that all of our expressions are in terms of $R$, and we have
eliminated the role of $N$. We now switch the orders of the sums and
the integral (which is legitimate since the sum is absolutely convergent) and get
\begin{equation} \label{int}
\int_{-\sqrt{R}}^{\sqrt{R}} \cdots \int_{-\sqrt{R}}^{\sqrt{R}}
\left( \sum_{d_1,\ldots,d_{n}\in \mathbf{G}_{R}
}\frac{g(d_1,\ldots,d_{n})}{|[d_1,\ldots,d_{n}]|}
\prod_{m=1}^{k}\mu(d_{m}) |d_{m}|^{-\frac{1+it_{m}}{R}} \right)
\prod_{m=1}^{k}\psi(t_{m}) dt_{1} \cdots dt_{k}
\end{equation}
Let us estimate the expression under the integration.
\begin{lemma}For every $I \subset \{1, \ldots, n \}, I \neq \emptyset$ let $c_{I}=1$ if $J_{i}=J_{i'}$ for every $i, i' \in I$, and 0 otherwise.
Then for every $t_1, \ldots, t_{n} \in [-\sqrt{R}, \sqrt{R}]$ we
have
\begin{equation} \label{eq4} \sum_{d_1,\ldots,d_{n}
}\frac{g(d_1,\ldots,d_{n})}{[d_1,\ldots,d_{n}]}
\prod_{m=1}^{n}\mu(d_{m})|d_{m}|^{-\frac{1+it_{m}}{R}}=GH(1+o_{R \rightarrow \infty}(1)) \left( \frac{\log q}{R} \right)^{r}
\prod_{\substack{I \subset \{1,\ldots,n \},\\ I \neq \emptyset}} \left(\sum_{m
\in I}(1+it_{m})\right)^{(-1)^{|I|+1}c_{I}}
\end{equation}
\end{lemma}
\begin{proof}
Recall that in the expression on the left hand side of (\ref{eq4}), only square-free $d_1,\ldots,d_{n}$ are involved. We have that $g(d_1,\ldots,d_{n})$ always takes on two values 0 and 1. More precisely, by the Chinese
remainder theorem, $g(d_1,\ldots,d_{n})=\prod_{P}c_{P,\{i:P|d_{i}\}}$, where the
$c_{P,I}$ are ``local factors'' defined by $c_{P,I}=\sharp \{
\d(f)<\d(P): P|f+J_{i} \textrm{ for every } i\in I\} $ for every $I
\subset \{1, \ldots, n \}, I \neq \emptyset$. We have the following
explicit formula:
$$c_{P,I}=
\left\{
                                         \begin{array}{ll}
                                           1, & \hbox{if $J_{i} \equiv J_{i'} \pmod{P}$ for every $i,i' \in I$ ;} \\
                                           0, & \hbox{otherwise.}
                                         \end{array}
                                       \right.$$
Let $\mathcal{S}$ be the set of all irreducible divisors of
$\Delta$. Then for $P$ outside of $\mathcal{S}$, we have
$c_{P,I}=c_{I}$.  The left hand side of (\ref{eq4}) can factor as
$$\prod_{P} \left( 1 - \sum_{I \subset \{1,\ldots,n \}, I \neq \emptyset} (-1)^{|I|+1} \frac{c_{P,I}}{|P|^{1+\sum_{j \in I}\frac{1+it_{m}}{R}}} \right)
$$
which is an Euler product treated in Lemma \ref{euler}. We know from
Lemma \ref{euler} that it is equal to
$$
GH(1+o_{R \rightarrow \infty}(1))\prod_{I \subset \{1, \ldots, k\}, I \neq
\emptyset}\left(\sum_{m \in I}\frac{1+it_{m}}{R} \log
q\right)^{(-1)^{|I|+1}c_{I}}
$$ where $G$ is the
arithmetic factor $$G=\prod_{P} \left(1+\sum_{I \neq \emptyset}
\frac{(-1)^{|I|}c_{P,I}}{|P|} \right) \left(1-\frac{1}{|P|}
\right)^{\sum_{I \neq \emptyset}(-1)^{|I|c_{I}}}$$
Let us verify that this is indeed the same expression for $G$ claimed at the beginning.
\begin{claim}
$\sum_{I \subset \{1, \ldots, n\}, I \neq \emptyset}(-1)^{|I|}c_{I}=-r$
\end{claim}
Indeed, if $a_1, \ldots, a_{r}$ are the multiplicities of $J_1,
\ldots, J_{n}$, then $$\sum_{I \subset \{1, \ldots, n\}, I \neq \emptyset}(-1)^{|I|}c_{I}=\sum_{s=1}^{r}
\sum_{j=1}^{a_{s}} (-1)^{j} \binom{a_{s}}{j} =-r$$
\begin{claim}
$\sum_{I \subset \{1, \ldots, n\}, I \neq
\emptyset}(-1)^{|I|c_{P,I}}=-\alpha_{P}$
\end{claim}
This follows from exactly the same observation as the previous
claim. Thus $G=\prod_{P} \left(
1-\frac{1}{|P|} \right)^{-r}\left(1-\frac{\alpha_{P}}{|P|} \right)$ and the lemma follows.
\end{proof}
By integrating over all $t_{1}, \ldots, t_{n}\in
[-\sqrt{R},\sqrt{R}]$, we see that the expression in (\ref{int}) is
equal to
\begin{eqnarray}
& & \left( \frac{\log q}{R} \right)^{r}G H(1+o(1))\int_{-\sqrt{R}}^{\sqrt{R}} \cdots
\int_{-\sqrt{R}}^{\sqrt{R}} \prod_{I \subset \{1, \ldots,n\}, I \neq
\emptyset}\left( \sum_{m \in I} (1+it_{m})
\right)^{(-1)^{|I|+1}c_{I}} \prod_{m=1}^{n}\psi(t_{m}) dt_{1} \cdots
dt_{n} \nonumber \\
&=& \left( \frac{\log q}{R} \right)^{r}G H(1+o(1))\int_{-\infty}^{\infty} \cdots
\int_{-\infty}^{\infty} \prod_{I \subset \{1, \ldots,n\}, I \neq
\emptyset}\left( \sum_{m \in I} (1+it_{m})
\right)^{(-1)^{|I|+1}c_{I}} \prod_{m=1}^{n}\psi(t_{m}) dt_{1} \cdots
dt_{n} \nonumber
\end{eqnarray}
since $\psi$ decreases rapidly. Thus we have proved the estimate
$(\ref{gye})$, with
$$C=\int_{-\infty}^{\infty} \cdots
\int_{-\infty}^{\infty} \prod_{I \subset \{1, \ldots,k\}, I \neq
\emptyset}\left( \sum_{m \in I} (1+it_{m})
\right)^{(-1)^{|I|+1}c_{I}} \prod_{m=1}^{k}\psi(t_{m}) dt_{1} \cdots
dt_{k}$$ This expression can be simplified a little bit. Let
$a_{1},\ldots,a_{r}$ be the multiplicities of $J_1,\ldots,J_{n}$,
then we have $C=\prod_{s}C_{a_{s}}$, where
\begin{equation} \label{c}
C_{a}=\int_{-\infty}^{\infty} \cdots
\int_{-\infty}^{\infty} \prod_{I \subset \{1, \ldots,a\}, I \neq
\emptyset}\left( \sum_{m \in I} (1+it_{m}) \right)^{(-1)^{|I|+1}}
\prod_{m=1}^{a}\psi(t_{m}) dt_{1} \cdots dt_{a}
\end{equation}
In our applications we will be able to compute these constants explicitly
in terms of $\chi$.
\end{proof}

\section{The correlation condition} \label{S10}
In this section we prove the correlation condition. As mentioned
before, we will need a variant of Proposition \ref{gy}:
\begin{proposition} \label{gy2}
Suppose $J_1, J_2, \ldots, J_{n} \in \Fq[t]$, not necessarily
distinct. Let $\Delta=\Delta(J_1,\ldots,J_{n})$ be defined as before in Proposition \ref{gy}. Then we have the asymptotic formula
\begin{equation}\label{gye2}
\sum_{f \in \GN}
\Lambda_{R}(W(f+J_1)+1)\cdots\Lambda_{R}(W(f+J_{n})+1)=CGH (1+o_{n}(1))q^{N} \left(\frac{\log q}{R}\right)^{r}
\end{equation}
where $G=\prod_{\d(P)<w} \left(1-\frac{1}{|P|} \right)^{-r} \prod_{\d(P) \geq w}\left(1-\frac{\alpha_{P}}{|P|} \right) \left(1-\frac{1}{|P|} \right)^{-r}, H=\prod_{P|\Delta}\left(1+O_{n}\left(\frac{1}{|P|} \right) \right)$, and $C$ is the same constant as in Proposition \ref{gy}.
\begin{proof}
The proof follows along the lines of that of Proposition \ref{gy}.
The only difference is that we have a different formula for local
factors $c_{P,I}=\sharp \{\d(f)<\d(P): P| W(f+J_{i}) + 1 \forall i
\in I) \}$: For any $I \subset \{1, \ldots, n \}, I \neq \emptyset$,
we have
$$c_{P,I}= \left\{
                                         \begin{array}{ll}
                                           1, & \hbox{if $\d(P)\geq w$ and $J_{i} \equiv J_{i'} \pmod{P}$ for every $i,i' \in I$ ;} \\
                                           0, & \hbox{otherwise.}
                                         \end{array}
                                       \right.$$
By incorporating this change into the proof, we will find the desired expression for $G$.
\end{proof}

\end{proposition}
\begin{proof}[Proof of the correlation condition]
We are interested in expressions of the form
$$\E \left( \nu(f+h_1)\cdots \nu(f+h_{l}) | f \in \GN \right)$$
where $h_1, \ldots, h_{l} \in \GN$ and the number of forms $l$ is
bounded by $l_0$ which depends only on $k$. Recall that our goal is
to find a function $\tau$ on $\GN$ such that
\begin{equation}\label{eq5}
\E \left( \nu(f+h_1)\cdots \nu(f+h_{l}) | f \in \GN \right) \leq \sum_{1 \leq 1 \leq j \leq l} \tau(h_{i}-h_{j}) \end{equation}
Moreover, for every $1\leq p<\infty$,
\begin{equation}\label{eq6}
\E(\tau(f)^{p})=O_{p}(1)
\end{equation}

In the event where two of the $h_{i}$ are equal, we bound $\E \left( \nu(f+h_1)\cdots \nu(f+h_{l}) | f \in \GN \right)$ by $\|\nu\|_{\infty}^{l}=q^{O(N/\log N)}$, thanks to Lemma \ref{divisor}. By choosing $\tau(0)=q^{O(N/\log N)}$ then clearly the inequality \ref{eq5} is satisfied. Moreover, since $q^{O(N/\log N)}=O_{\epsilon}(q^{N \epsilon})$ for every $\epsilon>0$, the addition of $q^{O(N/\log N)}$ to $\tau(0)$ does not affect the boundedness of $\E(\tau^{p})$ for every $p>1$.

Therefore, we have to find a function $\tau \in L^{p}$ for every $p>1$ so that the inequality (\ref{eq5}) is satisfied when all the $h_{i}$ are distinct. From the definition of $\nu$, we have
\begin{eqnarray}
& & \E \left( \nu(f+h_1)\cdots \nu(f+h_{l}) | f \in \GN \right) \nonumber \\
&=& R^{l} \left( \frac{\Phi(W)}{|W|} \right)^{l} \E \left( \Lambda_{R}(W(f+h_1)+1)^2\cdots \Lambda_{R}(W(f+h_{l})+1)^2 | f \in \GN \right) \nonumber
\end{eqnarray}
Thanks to Proposition $\ref{gy2}$, we know that $$\E \left( \Lambda_{R}(W(f+h_1)+1)^2\cdots \Lambda_{R}(W(f+h_{l})+1)^2 | f \in \GN \right)=CGH (1+o(1)) \left( \frac{\log q}{R} \right)^{l}$$
where
\begin{eqnarray}
G &=& \prod_{\d(P) < w}\left(1-\frac{1}{|P|} \right)^{-l} \prod_{\d(P) \geq w} \left(1-\frac{l}{|P|} \right) \left(1-\frac{1}{|P|}\right)^{-l} \nonumber \\
&=& \left(\frac{\Phi(W)}{W}\right)^{-l}\prod_{\d(P) \geq w} \left(1-\frac{l}{|P|} \right) \left(1-\frac{1}{|P|}\right)^{-l} \nonumber
\end{eqnarray}

Since $\prod_{P} \left(1-\frac{l}{|P|} \right) \left(1-\frac{1}{|P|}\right)^{-l}=\prod_{P} \left(1+O_{l}\left(\frac{l}{|P|^2} \right) \right)$ converges absolutely, and since $w \rightarrow \infty$, we have that $G=\left(\frac{\Phi(W)}{W}\right)^{-l}(1+o(1))$.

Let us compute $C$ explicitly. In this case, each $h_{i}$ has multiplicity 2, hence equation (\ref{c}) gives us:
\begin{eqnarray}
C&=& \left(\int_{-\infty}^{\infty} \int_{-\infty}^{\infty} \frac{(1+it_1)(1+it_2)}{2+it_1+it_2} \psi(t_1)\psi(t_2)
dt_1 dt_2\right)^{l} \nonumber \\
&=& \left(\int_{-\infty}^{\infty} \int_{-\infty}^{\infty} (1+it_1)(1+it_2)\psi(t_1)\psi(t_2) \log q \left( \int_{0}^{\infty}  q^{-(2+it_1+it_2)x} dx \right)
dt_1 dt_2\right)^{l} \nonumber \\
&=& \left(\int_{0}^{\infty} \log q \left(\int_{-\infty}^{\infty}q^{-(1+it)x}\psi(t) dt \right)^2 \right)^{l} \nonumber \\
&=& \left( \int_{0}^{\infty} \log q \left( \frac{\chi'(x)}{\log q} \right)^2 dx \right)^{l}\nonumber \\
&=& (\log q)^{-l} \nonumber
\end{eqnarray}

Therefore, $$\E \left( \nu(f+h_1)\cdots \nu(f+h_{l}) | f \in \GN \right) = (1+o(1))H$$
Recall that $H = \prod_{P | \Delta} \left( 1+ O_{l}\left( \frac{1}{|P|} \right) \right)$, where $\Delta=\prod_{1 \leq i \leq j \leq m}(h_{i}-h_{j})$. Let us bound $(1+o(1))H$ by $\exp\left( M \sum_{P| \Delta} \frac{1}{|P|} \right)$, where $M$ is a constant depending only on $l_0$, hence on $k$.

For $f \neq 0$, put $\tau(f)=\exp \left( K \sum_{P|f}\frac{1}{|P|}
\right)$ for $K$ sufficiently large depending on $M$, then clearly
the inequality (\ref{eq5}) is satisfied. The only thing left to verify is
the inequality (\ref{eq6}). By Lemma \ref{tz}
we have
\begin{eqnarray}
\E(\tau(f)^{p}|\GN) &=& \E \left( \exp \left( pK \sum_{P|f}\frac{1}{|P|}
\right) \Big| f \in \GN \right) \nonumber \\
&\ll_{K,p}& \E \left( \sum_{P|f}
\frac{\log^{Kp}|P|}{|P|} \Big| f \in \GN \right) \nonumber
\end{eqnarray}
For every $P$, the number of $f \in \GN$ that is divisible by $P$ is at most $\frac{q^{N}}{|P|}$. Thus
\begin{eqnarray}
\E \left( \sum_{P|f} \frac{\log^{Kp}|P|}{|P|} \Big| f \in \GN \right) &\leq & \frac{1}{q^{N}} \sum_{P} \frac{q^{N}}{|P|}\frac{\log^{Kp}|P|}{|P|} \nonumber \\
&=& \sum_{P}\frac{\log^{Kp}|P|}{|P|^2} = O_{K,p}(1) \nonumber
\end{eqnarray}
as required.
\end{proof}

\section{The linear forms condition} \label{S11}
In this section we prove the linear forms condition. For this condition, the following variant of Proposition \ref{gy} is needed:
\begin{proposition} \label{gy3}
Let $\psi_1, \ldots, \psi_{m}$ are $m$ non-zero linear forms in $n$ variables $(\GN)^{n} \rightarrow \Fq[t]$, not necessarily distinct, of the form
$$\psi_{i}(\mathbf{f})=\sum_{j=1}^{n}L_{ij} f_{j}+b_{i}$$
for every $\mathbf{f}=(f_1,\ldots,f_{n}) \in (\GN)^{n}$, where
\begin{itemize}
  \item $m \leq m_{0}, n \leq n_{0}$, where $m_0, n_0$ are constants depending on $k$.
  \item The coefficients $L_{ij}, b_{i} \in \Fq[t]$ and $\d(L_{ij})<w/2$ for every $i,j$.
  \item For any two $i,i'=1, \ldots, m$, either the two vectors $(L_{ij})_{j=1,\ldots,n}$ and $(L_{i'j})_{j=1,\ldots,n}$ are not proportional (over $\Fq(t)$), or they are identical and $\psi_{i}=\psi_{i'}$.
\end{itemize}
Let $r$ be the number of distinct forms in $\psi_1, \ldots, \psi_{m}$.
Then (assuming that $R= \alpha N$ and $\alpha$ is sufficiently small depending only on $m_0,n_0$) we have the following asymptotic formula as $N\rightarrow \infty$:
\begin{equation} \label{gye3}
\sum_{\mathbf{f} \in (\GN)^{n}} \Lambda_{R}(W\psi_1(f)+1) \cdots \Lambda_{R}(W\psi_{m}(f)+1) = C(1+o(1))\left(\frac{\Phi(W)}{W}\right)^{-r}\left( \frac{\log q}{R} \right)^{r}q^{Nn}
\end{equation}
where $C$ is a computable constant (depending only on $\chi$ and the multiplicities of the $\psi_{i}$).
\end{proposition}
\begin{proof}
The proof is similar to that of Proposition \ref{gy}. The left hand side of ($\ref{gye3}$) is equal to
$$\sum_{d_1,\ldots,d_{m} \in \GN} \left( \prod_{i=1}^{m} \mu (d_{i}) \chi \left( \frac{\d(d_{i})}{R} \right) \right) \frac{g(d_1,\ldots,d_{m})}{|[d_1, \ldots, d_{m}]|^{n}}q^{Nt}$$
where $g(d_1,\ldots,d_{m})$ is the number of solutions $\mathbf{f} \in (G_{|[d_1,\ldots,d_{m}]|})^{n}$ to the system of congruences $d_{i} | W\psi_{i}(\mathbf{f})+1$ for every $i=1,\ldots,m$.

Again, by the Chinese Remainder Theorem,  $g(d_1,\ldots,d_{m})$ factors as $\prod_{P}c_{P,\{ i:P|d_{i}\} }$, where the local factors $c_{P,I}$ are defined by $c_{P,I}=\sharp \{ \mathbf{f} \in (\mathbf{G}_{\d(P)})^{n}: P| W\psi_{i}(\mathbf{f})+1 \textrm{ for all } i \in I\}$ for every $I \subset \{1, \ldots, m\}, I \neq \emptyset$.

Clearly if $\d(P)<w$ then $c_{P,I}=0$.
Let us compute $c_{P,I}$ when $\d(P) \geq w$. In particular $\gcd(P,W)=1$. The system of congruences $P| W\psi_{i}(\mathbf{f})+1 \textrm{ for all } \in I$ amounts to a system of $|I|$ equations
$\psi_{i}(\mathbf{f})=-W^{-1}$ for every $i \in I$, where the $\psi_{i}$ are now regarded as affine maps $(\mathbf{F}_{P})^{n} \rightarrow \mathbf{F}_{P}$, where $\mathbf{F}_{P}=\Fq[t]/(P)$.

Note that when regarded as forms on $\mathbf{F}_{P}$, the property that for any two forms $\psi_{i},\psi_{i'}$, either their homogeneous parts are not proportional or they are identical, is still preserved. Indeed, this is obviously true if $n=1$. Suppose $n \geq 2$ and we have that for some $i,i'$, $\frac{L_{ij}}{L_{i'j}}=\frac{L_{ij'}}{L_{i'j'}}$ in $\mathbf{F}_{P}$ for any $j, j'=1, \ldots, n$. Then $P$ divides  $L_{ij}L_{i'j'}-L_{i'j}L_{ij'}$. Since $\d(L_{ij})<w/2$ for every $i,j$, this means that $L_{ij}L_{i'j'}-L_{i'j}L_{ij'}=0$. Therefore, the two vectors $(L_{ij})_{j=1, \ldots, n}$ and $(L_{i'j})_{j=1, \ldots, n}$ are indeed proportional over $\Fq(t)$, so that the forms $\psi_{i}$ and $\psi_{i'}$ are identical.

Being the number of solutions to a system of non-trivial linear equations over $\mathbf{F}_{P}$, $c_{P,I}$ is either 0 or a power of $|P|$ not exceeding $|P|^{n-1}$. By the assumption made on the $\psi_{i}$, $c_{P,I}=|P|^{n-1}$ if and only if all the forms in $I$ are identical. Otherwise, $c_{P,I}=O(|P|^{n-2})$.

Incorporating this change, we have that
$$ \sum_{d_1,\ldots,d_{m}} \frac{g(d_1,\ldots,d_{m})}{|[d_1, \ldots, d_{m}]|^{n}} \prod_{j=1}^{m} \mu(d_{j}) |d_{j}|^{-\frac{1+it_{j}}{R}} = \prod_{P} \left( 1 - \sum_{I \subset \{1,\ldots,m \}, I \neq \emptyset} (-1)^{|I|+1}\frac{c_{P,I}}{|P|^{n}}|P|^{\sum_{j \in I}-\frac{1+it_{j}}{R}} \right)
$$
By the above computation, this is equal to
\begin{eqnarray}
& & \prod_{\d(P)\geq w}\left( 1-\sum_{I} \frac{(-1)^{|I|+1}}{|P|^{1+\sum_{j \in I}\frac{1+it_{j}}{R}}}+O\left(\frac{1}{|P|^2}\right)\right) \nonumber \\
&=& (1+o(1))\prod_{\d(P)\geq w}\left( 1-\sum_{I}\frac{(-1)^{|I|+1}}{|P|^{1+\sum_{j \in I}\frac{1+it_{j}}{R}}}\right) \nonumber
\end{eqnarray}
where the sums are taken over all non-empty subsets $I$ of $\{1, \ldots, m\}$ such that the $\psi_{i}, i \in I$, are all identical. We now understand the use of the $W$-trick: it helps absorb the terms $O\left(\frac{1}{|P|^2}\right)$, which in turn comes from congruence relations between coefficients of the $\psi_{i}$. Invoking Lemma \ref{euler}, we see that the above product is equal to
$$G(1+o(1))\prod_{I} \left( \sum_{j \in I}(1+it_{j}) \frac{\log q}{R} \right)^{(-1)^{|I|+1}}$$
where $G$ is the arithmetic factor $G=\prod_{\d(P) \geq w}\left( 1+ \sum_{I}\frac{ (-1)^{|I|}}{|P|} \right) \left(1-\frac{1}{|P|}\right)^{\sum_{I}(-1)^{I}}$, the product being taken over all non-empty subsets $I$ of $\{1, \ldots, m\}$ such that the $\psi_{i}, i \in I$, are all identical.
It is easy to see that $\sum_{I}(-1)^{|I|}=r$, so that $G=(1+o(1))\left(\frac{\Phi(W)}{|W|}\right)^{-r}$.

We then proceed as in the proof of Proposition \ref{gy} and see that
$$\sum_{\mathbf{f} \in (\GN)^{n}} \Lambda_{R}(W\psi_1(f)+1) \cdots \Lambda_{R}(W\psi_1(f)+1) = C(1+o(1))\left(\frac{\Phi(W)}{W}\right)^{-r}\left( \frac{\log q}{R} \right)^{-r}q^{Nn}$$
where
\begin{equation}
C = \int_{-\infty}^{\infty}\cdots \int_{-\infty}^{\infty} \prod_{I} \left( \sum_{j \in I}(1+it_{j})\right)^{(-1)^{|I|+1}} \prod_{j=1}^{m}\psi(t_{j}) dt_1 \cdots dt_{m} \nonumber
\end{equation}
 the product being taken over all non-empty subsets $I$ of $\{1, \ldots, m\}$ such that the $\psi_{i}, i \in I$, are all identical. Similarly to the constant $C$ in Proposition \ref{gy}, $C$ factors as $\prod_{s=1}^{r}C_{a_{s}}$, where $a_1, \ldots, a_{r}$ are multiplicities of the $\psi_{i}$, and
$$C_{a}=\int_{-\infty}^{\infty} \cdots
\int_{-\infty}^{\infty} \prod_{I \subset \{1, \ldots,a\}, I \neq
\emptyset}\left( \sum_{j \in I} (1+it_{j}) \right)^{(-1)^{|I|+1}}
\prod_{m=1}^{a}\psi(t_{j}) dt_{1} \cdots dt_{a}$$
\end{proof}
\begin{proof}[Proof of the linear forms condition]
We are interested in expressions of the form $$\E
\left(\nu(\psi_1(\mathbf{f}))\cdots \nu(\psi_{m}(\mathbf{f}))|
\mathbf{f} \in (\FqN)^{n} \right)$$ where the $\psi_{i}(\mathbf{f})=\sum_{j=1}^{n}L_{ij}f_{j}+b_{i}$ are linear
form in $n$ variables, where $m \leq m_0, n \leq n_0$,  no two
homogenous parts are proportional, and the coefficients $L_{ij}$ are in the set $\{\frac{P}{Q}: \d(P), \d(Q) <k\}$. Recall that we want to bound
these expressions by $1 + o(1)$.

By the definition of $\nu$, this expression is equal to
\begin{equation} \label{e}
\left(\frac{\Phi(W)}{|W|} \right)^{m} R^{m} \E \left( \Lambda_{R}(W\psi_1(\mathbf{f})+1)^2 \cdots \Lambda_{R}(W\psi_{m}(\mathbf{f})+1)^2 | \mathbf{f} \in (\FqN)^{n} \right)
\end{equation}

Our first reduction is to replace the assumption that all the coefficients $L_{ij}$ are in $\{ \frac{f}{g} | f, g \in \mathbf{G}_{k} \}$ by $L_{ij} \in \mathbf{G}_{M}$ for some sufficiently large $M$ depending on $k$. Indeed, via a change of variables $\mathbf{f} \mapsto \left( \prod_{h \in \mathbf{G}_{k}, h \textrm{ monic}} h \right) \mathbf{f} $, the $\psi_{i}$ become linear forms with coefficients in $\Fq[t]$ and of degrees still bounded $\d \left( \prod_{h \in \mathbf{G}_{k},h \textrm{ monic}} h \right)+k =\sum_{d=1}^{k-1}dq^{d}+k=M$.

We are tempted to apply Proposition \ref{gy3} right away. However, a priori the $\psi_{i}$ are linear forms from $(\FqN)^{n}$ to $\FqN$, which are different from the linear forms $\Psi_{i}$ from $(\Fq[t])^{n}$ to $\Fq[t]$ given by the same formula $\Psi_{i}(\mathbf{f})=\sum_{j=1}^{n}L_{ij}f_{j}+b_{i}$ for every $\mathbf{f}=(f_1,\ldots,f_{n})\in (\Fq[t])^{n}$. More precisely, $\psi_{i}(\mathbf{f})$ is the residue of $\Psi_{i}(\mathbf{f})$ upon division by $f_{N}$, the irreducible polynomial underlying $\FqN$.

To remedy this, let us divide $(\Fq[t])^{n}$ into $q^{nM}$ ``boxes'' such that for $\mathbf{f},\mathbf{f'}$ in the same box $B$, we have $\max_{i}|f_{i}-f'_{i}|< q^{N-M}$ (in other words, each box is a product of cylinders of radius $q^{M}$ in $\Fq[t]$). Then for $\mathbf{f},\mathbf{f'}$ in the same box, we have $\Psi_{i}(\mathbf{f})-\Psi_{i}(\mathbf{f'})<q^{N-M+M}=q^{N}$. This means that the residues of $\Psi_{i}(\mathbf{f})$ and $\Psi_{i}(\mathbf{f})$ upon division by $f_{N}$ are the same. In other words, for any box $B$, we have a formula for $\psi_{i}(\mathbf{f})$:
$$\psi_{i}(\mathbf{f})=\Psi_{i,B}(\mathbf{f})=\sum_{j=1}^{n}L_{ij}f_{j}+b_{i,B}$$
for every $\mathbf{f} \in B$, and $b_{i,B}$ depends alone on the box $B$.

We now rewrite the expression (\ref{e}) as
$$\left(\frac{\Phi(W)}{|W|} \right)^{m} R^{m} \frac{1}{q^{nM}} \sum_{B} \E \left( \Lambda_{R}(W\psi_{1,B}(\mathbf{f})+1)^2 \cdots \Lambda_{R}(W\psi_{m,B}(\mathbf{f})+1)^2 | \mathbf{f} \in B \right) $$

Note that for $N$ sufficiently large, we have $\d(L_{ij}) < M <w/2$. For each box $B$, Proposition \ref{gy3} tells us that
$$\E \left( \Lambda_{R}(W\psi_{1,B}(\mathbf{f})+1)^2 \cdots \Lambda_{R}(W\psi_{m,B}(\mathbf{f})+1)^2 | \mathbf{f} \in B \right)=C(1+o(1))\left(\frac{\Phi(W)}{W}\right)^{-m}\left( \frac{\log q}{R} \right)^{m}$$
Similarly to the calculations in the proof of the correlation condition, we see that $C= (\log q)^{m}$. Summing this up over all the $q^{M}$ boxes, the linear forms condition is therefore verified.
\end{proof}


\end{document}